\documentclass{amsart}
\usepackage{amsmath}
\usepackage{paralist}
\usepackage{url}
\usepackage{color}
\usepackage{graphicx}
\usepackage{subcaption}
\usepackage{multirow}
\usepackage{mathtools}
\usepackage{amssymb}
\usepackage{booktabs}
\usepackage{makecell}

\usepackage{mathrsfs}
\usepackage{xspace}
\usepackage{comment}
\usepackage{arydshln}
\usepackage{caption}
\usepackage{algorithm}
\usepackage{algpseudocode}
\usepackage{hyperref}
\usepackage{float}
\usepackage{placeins}
\usepackage{kotex}
\usepackage{tabularx}

\usepackage{tikz}
\usetikzlibrary{arrows}
\usepackage{pgfplots}
\pgfplotsset{compat=newest}

\newcommand{\RN}[1]{%
  \textup{\uppercase\expandafter{\romannumeral#1}}%
}

\newcommand{\bdes}{\begin{description}}
	\newcommand{\edes}{\end{description}}
\newcommand{\bal}{\begin{align}}
\newcommand{\eal}{\end{align}}
\newcommand{\bnum}{\begin{enumerate}}
	\newcommand{\enum}{\end{enumerate}}
\newcommand{\bit}{\begin{itemize}}
	\newcommand{\eit}{\end{itemize}}
\newcommand{\bea}{\begin{eqnarray}}
\newcommand{\eea}{\end{eqnarray}}
\newcommand{\be}{\begin{equation}}
\newcommand{\ee}{\end{equation}}
\newcommand{\baray}{\begin{array}}
	\newcommand{\earay}{\end{array}}
\newcommand{\bsry}{\begin{subarray}}
	\newcommand{\esry}{\end{subarray}}
\newcommand{\bca}{\begin{cases}}
	\newcommand{\eca}{\end{cases}}
\newcommand{\bcen}{\begin{center}}
	\newcommand{\ecen}{\end{center}}
\newcommand{\bbm}{\begin{bmatrix}}
	\newcommand{\ebm}{\end{bmatrix}}
\newcommand{\bmx}{\begin{matrix}}
	\newcommand{\emx}{\end{matrix}}
\newcommand{\bpm}{\begin{bmatrix}}
	\newcommand{\epm}{\end{bmatrix}}
\newcommand{\btab}{\begin{tabular}}
	\newcommand{\etab}{\end{tabular}}

\DeclareMathOperator{\bcirc}{bcirc}
\DeclareMathOperator{\unfold}{unfold}
\DeclareMathOperator{\fold}{fold}
\DeclareMathOperator{\fft}{fft}
\DeclareMathOperator{\ifft}{ifft}
\DeclareMathOperator{\blkdiag}{blkdiag}
\DeclareMathOperator{\trace}{trace}

\newtheorem{theorem}{Theorem}[section]

\newtheorem{lemma}[theorem]{Lemma}

\newtheorem{ass}[theorem]{Assumption}
\newtheorem{defi}[theorem]{Definition}
\theoremstyle{definition}

\numberwithin{equation}{section}

\newcolumntype{Y}{>{\centering\arraybackslash}X}

\pgfplotsset{every axis/.append style={tick label style={/pgf/number format/fixed},font=\scriptsize,ylabel near ticks,xlabel near ticks,grid=major}}

\pgfmathdeclarefunction{sumexp}{3}{%
  \begingroup%
  \pgfkeys{/pgf/fpu}
  \pgfmathsetmacro{\myx}{#1}%
  \pgfmathtruncatemacro{\myxmin}{#2}%
  \pgfmathtruncatemacro{\myxmax}{#3}%
  \pgfmathsetmacro{\mysum}{0}%
  \pgfplotsforeachungrouped\XX in {\myxmin,...,\myxmax}%
    {\pgfmathsetmacro{\mysum}{\mysum+exp(\XX)}}%
  \pgfmathparse{\mysum+exp(#1)}%
  \pgfmathfloattofixed\pgfmathresult
  \pgfmathsmuggle\pgfmathresult\endgroup%
}%

\begin{document}

\title[SD-TRK for Doubly Noisy Tensor Systems]
{A Spectrally Damped Tensor Randomized Kaczmarz Method for Doubly Noisy Tensor Systems}

\author[Nahyun Lee]{Nahyun~Lee$^*$}
\author[Jiyoung Choi]{Jiyoung~Choi$^*$}

\def\thefootnote{*}\footnotetext{All authors contributed equally.}

\address{Nahyun Lee, Jiyoung Choi,  Department of Mathematics, University of California San Diego, 9500 Gilman Drive, La Jolla, CA, USA, 92093.}
\email{nal003@ucsd.edu, jichoi@ucsd.edu}

\begin{abstract}
Tensor randomized Kaczmarz (TRK) methods are efficient row-action solvers for
tensor linear systems under the t-product framework. We study their behavior
under a doubly noisy perturbation model. In this model, both the system tensor and
the right-hand side tensor are corrupted. We first analyze standard TRK and derive
an expected error recursion with two terms. One term is contractive, and the other
is a persistent perturbation term. This explains the noise-limited and
semi-convergent behavior that can occur when the observed tensor system is
inconsistent. We then introduce a spectrally damped tensor randomized Kaczmarz
method (SD-TRK). We
prove an expected error recursion for SD-TRK that separates error propagation from
noise injection. The bound makes explicit a speed--robustness trade-off. We also
give an FFT-based implementation that applies the damped update slice-wise in the
Fourier domain. This implementation allows frequency-dependent damping parameters
in practice. Numerical experiments on synthetic tensor systems illustrate the stabilization
behavior of SD-TRK relative to standard TRK in noisy and ill-conditioned settings.
We also include a two-pass image reconstruction comparison under the same
noisy reconstruction pipeline.
\end{abstract}

\maketitle

\section{Introduction}

The Kaczmarz method~\cite{Kaczmarz1937} is a classical row-action method for solving linear systems. 
Its randomized version has been widely studied as an iterative solver for large-scale overdetermined systems. 
Each iteration uses only one row of the system, and the method has an expected linear convergence guarantee in the consistent setting \cite{StrohmerVershynin2009}. 
This idea has led to many randomized
projection methods for linear systems and least-squares problems, including noisy variants, randomized extended Kaczmarz methods for least-squares problems, and general sketch-and-project frameworks
\cite{GowerRichtarik2015,Needell2010,ZouziasFreris2013}. 

Noise changes the behavior of randomized Kaczmarz methods. 
Instead of converging exactly to the noiseless solution, the iterates may approach a noise-dependent error neighborhood or exhibit semi-convergence \cite{Needell2010}. 
Doubly noisy matrix systems have also been analyzed recently \cite{BBDLM2024}.
In such systems, both the coefficient matrix and the right-hand side are perturbed. 
This model is relevant when the operator is built from uncertain measurements or approximate physical models.

Many modern data sets and inverse problems are tensor-valued rather than matrix-valued. 
Examples include color images, videos, multi-channel measurements, and other multidimensional arrays \cite{CichockiEtAl2015,KoldaBader2009,SidiropoulosEtAl2017}. 
Tensor methods can preserve multi-way structure that may be lost under naive matricization. 
In this paper, we study third-order tensor linear systems under the tensor t-product framework.
The t-product was introduced by Kilmer and Martin \cite{KilmerMartin2011}. 
It was further developed in
\cite{KernfeldKilmerAeron2015,KilmerBramanHaoHoover2013}. 
This framework gives a matrix-mimetic algebra for third-order tensors through block-circulant representations and Fourier-domain slice-wise multiplication. 
Ma and Molitor used this framework to introduce the tensor randomized Kaczmarz method (TRK) and proved expected linear convergence for consistent tensor linear systems \cite{MaMolitor2022}. 
Recent work has also developed randomized Kaczmarz methods for t-product tensor systems with factorized operators \cite{CastilloEtAl2025}.

The existing TRK theory primarily addresses consistent tensor systems. In practical
applications, both the system tensor and the observation tensor may be corrupted.
The observed tensor system is then generally inconsistent and need not have the
noiseless target as an exact solution. Consequently, the standard TRK iteration is
no longer governed by the purely contractive mechanism that appears in the
noiseless theory. In this paper we formalize this doubly noisy model in
Subsection~\ref{sec:doubly_error}. We then analyze the resulting error recursion for
standard TRK and introduce a damped variant designed to reduce perturbation
amplification.

Our stabilization strategy combines relaxation with spectral damping. Relaxation
parameters have long been used in Kaczmarz-type methods for inconsistent systems
\cite{Censor1983,Nikazad2024}. Tikhonov-type regularization and spectral filtering
are standard tools for suppressing noise amplification in ill-conditioned inverse
problems \cite{Chung2011,Hansen1998,Hansen2006,Tikhonov1963}. Motivated by these
ideas, we introduce a spectrally damped tensor randomized Kaczmarz method
(SD-TRK). The method replaces the inverse normalization term in the TRK update by a
regularized inverse and also includes a relaxation parameter. In the Fourier-domain
implementation, damping can be chosen slice-wise. This allows different frequency
components to be stabilized with different strengths according to their local
conditioning.

The main contributions of this paper are as follows.

\begin{itemize}
\item We analyze standard TRK under a doubly noisy tensor model. In particular, we
derive an expected error recursion showing that the standard TRK error consists of 
a contractive term and a persistent perturbation term. The perturbation term is
amplified by a frequency-imbalance constant that reflects how the observed row slices 
behave across Fourier modes.

\item We introduce the spectrally damped TRK method for doubly noisy tensor systems.
The proposed update incorporates both a relaxation parameter \(\omega\) and a damping
parameter \(\lambda\), and it reduces to standard TRK when \(\omega=1\) and
\(\lambda=0\).

\item We prove an expected error recursion for SD-TRK. The resulting bound separates
the propagation of the previous error from the perturbation injected by the doubly
noisy model, thereby making explicit the speed--robustness trade-off controlled by
\(\omega\) and \(\lambda\).

\item We provide an FFT-based implementation of SD-TRK. The implementation avoids
forming the block-circulant matrix explicitly and applies the damped update
slice-wise in the Fourier domain, where frequency-dependent damping parameters
\(\lambda^{(k)}\) can be used in practice.

\item We also include a noise-scaling comparison with a damped reference horizon
and a two-pass image reconstruction comparison between standard TRK and SD-TRK
under the same reconstruction procedure.
\end{itemize}

The remainder of the paper is organized as follows. Section~\ref{sec:pre} reviews
the matrix randomized Kaczmarz method, the tensor t-product, and the standard
TRK method. Section~\ref{sec:error} analyzes the behavior of standard TRK in
noiseless and doubly noisy tensor systems. Section~\ref{sec:sdtrk} introduces
SD-TRK and establishes its expected error recursion. Section~\ref{sec:experiments} presents numerical experiments, including synthetic
tensor systems and an image reconstruction comparison. 
Section~\ref{sec:con} concludes the paper.

\section{Preliminaries}\label{sec:pre}

\subsection*{Notation}
Throughout this paper, scalars and vectors are denoted by lowercase letters (e.g., $a$), matrices by capital letters (e.g., $A$), and third-order tensors by calligraphic capital letters (e.g., $\mathcal{A}$).
The set of integers $\{1, 2, \dots, m\}$ is denoted by $[m]$.
For a matrix $A$, we denote its $i$-th row by $A_{i:}$.
For a vector $a$, the Euclidean norm is denoted by $\|a\|$ or $\|a\|_2$, and the standard inner product is denoted by $\langle \cdot, \cdot \rangle$.
For a matrix $A$, $\|A\|_2$ denotes the spectral norm.
The conjugate transpose is denoted by the superscript asterisk (e.g., $A^*$).
For a matrix $A$, $\sigma_{\min}(A)$ denotes its smallest singular value and $\sigma_{\max}(A)$ denotes its largest singular value.
For a third-order tensor $\mathcal{A}$,
$\mathcal{A}_{i::}$ denotes the $i$-th horizontal slice (row slice), and $\mathcal{A}_{::k}$ denotes the $k$-th frontal slice.
The Frobenius norm of a tensor $\mathcal{A}$ is denoted by $\|\mathcal{A}\|_F$ and calculated as
\begin{equation*}
    \|\mathcal{A}\|_F = \sqrt{\sum_{i,j,k} |\mathcal{A}_{ijk}|^2}.
\end{equation*}
We write \(\fft(\cdot,[\,],3)\) for the one-dimensional discrete Fourier transform applied to every tube fiber along the third dimension, with transform length \(n_3\); this notation specifies the transform direction and is independent of the implementation language.
For a tensor \(\mathcal A\in\mathbb C^{n_1\times n_2\times n_3}\), we write
\[
\hat{\mathcal A}:=\fft(\mathcal A,[\,],3).
\]
The inverse transform is denoted by \(\ifft(\cdot,[\,],3)\).
The \(k\)-th frontal slice of \(\hat{\mathcal A}\) is denoted by
\(\hat A^{(k)}\), and the \(i\)-th row of this slice by \(\hat A_{i:}^{(k)}\).
Thus, for the observed tensor \(\tilde{\mathcal A}\),
\(\hat{\tilde A}_{i:}^{(k)}\) denotes the \(i\)-th row of the \(k\)-th frontal
slice of \(\hat{\tilde{\mathcal A}}\).
Fourier-domain quantities are generally complex-valued, even when the original
tensors are real-valued.

\subsection{Matrix randomized Kaczmarz (MRK)}

The classical matrix randomized Kaczmarz method (MRK) for solving a linear system $Ax = b$ iteratively projects the current iterate onto the hyperplane defined by a single row of the system. At each iteration $t$, a row index $i_t$ is selected with probability proportional to $\|A_{i_t:}\|^2$, and the update is given by \cite{StrohmerVershynin2009}:
\begin{equation*}
x_{t+1} = x_t - A_{i_t:}^* \frac{\langle A_{i_t:}, x_t \rangle - b_{i_t}}{\|A_{i_t:}\|^2},
\end{equation*}
where $A_{i_t:}^*$ denotes the conjugate transpose of the selected row. 
This row-wise structure makes the method well suited for large-scale problems, since each update uses only the selected row and does not require forming dense matrix factorizations.

\subsection{Tensor linear algebra}

Tensors appear in a wide range of applications, and manipulating them in their native form, rather than simply reshaping them into matrices, can retain important structural information and often yield computational benefits. 
However, many core notions and results from classical linear algebra do not extend to tensors in a straightforward way. 

To overcome this difficulty, we focus on tensor linear systems formulated under the tensor t-product. 
The t-product, introduced by Kilmer and Martin \cite{KilmerMartin2011}, is a bilinear tensor operation that carries many familiar matrix-algebra concepts and properties into the tensor setting. 
Notably, it provides a natural notion of orthogonality for tensors, which plays a central role in the analysis of the tensor randomized Kaczmarz (TRK) method.

To define the t-product, we first introduce a preliminary definition.

\begin{defi}
For \(\mathcal A \in \mathbb C^{n_1\times n_2\times n_3}\), let
\(\bcirc(\mathcal A)\) denote the block-circulant matrix
\begin{equation*}
\bcirc(\mathcal A) =
\left[
\begin{array}{ccccc}
A^{(1)} & A^{(n_3)} & A^{(n_3-1)} & \cdots & A^{(2)} \\
A^{(2)} & A^{(1)} & A^{(n_3)} & \cdots & A^{(3)} \\
\vdots & \vdots & \vdots & \ddots & \vdots\\
A^{(n_3)} & A^{(n_3-1)} & A^{(n_3-2)} & \cdots & A^{(1)}
\end{array}
\right]
\in \mathbb C^{n_1n_3 \times n_2 n_3},
\end{equation*}
where \(A^{(i)}=\mathcal A_{::i}\) for \(i=1,\ldots,n_3\).
\end{defi}

To express the tensor operation induced by \(\bcirc(\mathcal A)\) using standard
matrix multiplication, we reshape a third-order tensor into a block column matrix
by stacking its frontal slices. This is done via the \(\unfold\) operator, whose
inverse is \(\fold\):
\begin{equation*}
\unfold(\mathcal A)=
\left[
\begin{array}{c}
\mathcal A_{::1}\\
\vdots\\
\mathcal A_{::n_3}
\end{array}
\right],
\qquad
\fold(\unfold(\mathcal A))=\mathcal A.
\end{equation*}

\begin{defi}
For given tensors \(\mathcal A\in\mathbb C^{n_1\times n_2\times n_3}\) and
\(\mathcal B\in\mathbb C^{n_2\times n_4\times n_3}\), the t-product
\(\mathcal A\mathcal B\) is defined as
\begin{equation*}
\mathcal A\mathcal B
=
\fold\!\left(\bcirc(\mathcal A)\,\unfold(\mathcal B)\right)
\in \mathbb C^{n_1\times n_4\times n_3}.
\end{equation*}
\end{defi}

The block-circulant and unfolding representations above are useful for theoretical
analysis because they express the t-product using standard matrix multiplication.
In implementation, however, we do not explicitly construct these large matrices.
Instead, we compute the t-product in the Fourier domain using the Fast Fourier
Transform (FFT). This preserves the native tensor structure and reduces the
computational cost. Further details are given in Appendix~\ref{sec:appendix_fft}.

To introduce TRK in Subsection~\ref{sec:TRK}, we need a few further definitions in addition to those given above.

\begin{defi}
The $n \times n \times n_3$ identity tensor is the tensor whose first frontal slice is the $n \times n$ identity matrix and whose remaining entries are all zeros, and we denote it by $\mathcal{I}$.
\end{defi}

\begin{defi}
The conjugate transpose of a tensor $\mathcal{A} \in \mathbb{C}^{n_1 \times n_2 \times n_3}$ is denoted by $\mathcal{A}^*$ and is produced by taking the conjugate transpose of all frontal slices and reversing the order of the frontal slices $2,\ldots,n_3$.
\end{defi}

For computing orthogonal projections and least-squares solutions, the conjugate transpose plays a central role. 
It is defined so that $(\mathcal{A}\mathcal{B})^* = \mathcal{B}^*\mathcal{A}^*$ holds.
With the identity tensor defined, we can now rigorously define invertibility under the $t$-product.

\begin{defi}
A tensor $\mathcal{A}$ is invertible if there exists an inverse tensor $\mathcal{A}^{-1}$ such that
\begin{equation*}
    \mathcal{A}\mathcal{A}^{-1} = \mathcal{A}^{-1}\mathcal{A} = \mathcal{I}.
\end{equation*}
\end{defi}

Orthogonality is crucial for the convergence analysis of Kaczmarz methods, since the algorithm can be viewed as performing successive orthogonal projections onto solution spaces.

\begin{defi}
A tensor $\mathcal{Q} \in \mathbb{C}^{n_1 \times n_1 \times n_3}$ is orthogonal if
\begin{equation*}
    \mathcal{Q}\mathcal{Q}^* = \mathcal{Q}^*\mathcal{Q} = \mathcal{I}.
\end{equation*}
\end{defi}

\subsection{Tensor randomized Kaczmarz (TRK)}\label{sec:TRK}

Ma and Molitor \cite{MaMolitor2022} introduced a Kaczmarz-type iterative method for solving $t$-product tensor linear systems
\begin{equation*}
    \mathcal{A}\mathcal{X} = \mathcal{B},
\end{equation*}
where $\mathcal{A} \in \mathbb{C}^{n_1 \times n_2 \times n_3}$, $\mathcal{X} \in \mathbb{C}^{n_2 \times n_4 \times n_3}$, and $\mathcal{B} \in \mathbb{C}^{n_1 \times n_4 \times n_3}$.
The TRK update for the tensor linear system is
\begin{equation}\label{eq:TRK}
    \mathcal{X}^{t+1}=\mathcal{X}^t - \mathcal{A}_{i_t::}^{*}(\mathcal{A}_{i_t::}\mathcal{A}_{i_t::}^{*})^{-1}(\mathcal{A}_{i_t::}\mathcal{X}^t-\mathcal{B}_{i_t::}),
\end{equation}
where the index $i_t \in [n_1]$ is selected according to a probability distribution on the row indices at each iteration $t$.

To ensure that the TRK update and the corresponding tensor projection operators discussed later are rigorously well-defined throughout this paper, we make the following fundamental assumption.

\begin{ass} \label{assump:invertibility}
Whenever a row slice $\mathcal{R}_{i::}$ is used in a TRK-type update, the associated normalization tensor
\[
\mathcal{R}_{i::}\mathcal{R}_{i::}^*
\]
is assumed to be invertible.
\end{ass}

Here, $\mathcal{R}$ denotes the tensor that appears in the update, namely the system tensor in the noiseless setting and the observed tensor in the doubly noisy setting.

The TRK update \eqref{eq:TRK} can be viewed as a direct algebraic generalization of the classical MRK method. 
While MRK projects the iterate onto a hyperplane defined by a row vector, TRK projects the iterate onto a tensor subspace defined by a horizontal slice.
A key distinction lies in the normalization step. Scalar division in MRK is replaced by the inversion of a tube in TRK. Table~\ref{tab:mrk_vs_trk} summarizes this structural correspondence.

\begin{table}[ht]
    \centering
    \caption{Structural Correspondence between Matrix and Tensor Kaczmarz}
    \label{tab:mrk_vs_trk}
    \renewcommand{\arraystretch}{1.25}
    \begin{tabular}{c|c|c}
        \hline
        Component & MRK & TRK \\
        \hline
        Input Data & Row vector $A_{i_t:} \in \mathbb{C}^{1 \times n_2}$ & Row slice $\mathcal{A}_{i_t::} \in \mathbb{C}^{1 \times n_2 \times n_3}$ \\
        Operation & Inner product $\langle A_{i_t:},x_t\rangle$ & t-product $\mathcal{A}_{i_t::}\mathcal{X}^t$ \\
        Residual &  $\langle A_{i_t:},x_t\rangle - b_{i_t}$ &  $\mathcal{A}_{i_t::}\mathcal{X}^t - \mathcal{B}_{i_t::}$ \\
        Normalization & Scalar division $1/\|A_{i_t:}\|_2^{2}$ & Tube inverse $(\mathcal{A}_{i_t::}\mathcal{A}_{i_t::}^{*})^{-1}$ \\
        Update Direction & $A_{i_t:}^*$ & $\mathcal{A}_{i_t::}^*$ \\
        \hline
    \end{tabular}
\end{table}

The following section recalls the noiseless convergence guarantee for standard TRK
and then analyzes how the error recursion changes under the doubly noisy model
formalized in Subsection~\ref{sec:doubly_error}.

\section{Error Analysis of TRK}\label{sec:error}

In this section, we analyze the error behavior of TRK in the noiseless and doubly noisy settings.

\subsection{Standard TRK in the noiseless setting}

The TRK method is an effective solver for consistent tensor linear systems. Ma and Molitor \cite{MaMolitor2022} showed that the TRK algorithm converges linearly in expectation for a consistent tensor linear system as follows.

\begin{theorem}[Convergence of standard TRK \cite{MaMolitor2022}] \label{thm:standard_convergence}
Consider the consistent tensor linear system $\mathcal{A}\mathcal{X} = \mathcal{B}$. Let $\mathcal{X}^\star$ be the exact solution and $\mathcal{X}^t$ be the $t$-th iterate of the TRK algorithm. If each row slice index $i_t \in [n_1]$ is selected with probability $p_{i_t} = \frac{\|\mathcal{A}_{i_t::}\|_F^2}{\|\mathcal{A}\|_F^2}$, then the expected error satisfies:
\begin{equation}\label{ineq:TRK}
\mathbb{E}[\|\mathcal{X}^{t+1} - \mathcal{X}^\star\|_F^2] \le \left( 1 - \frac{\sigma_{\min}^2(\bcirc(\mathcal{A}))}{n_3 \|\mathcal{A}\|_F^2} \right) \mathbb{E}[\|\mathcal{X}^t - \mathcal{X}^\star\|_F^2].
\end{equation}
\end{theorem}

Theorem~\ref{thm:standard_convergence} provides the theoretical guarantee for the convergence of standard TRK in the absence of perturbations. 
Note that
\[
\sigma_{\min}^2(\bcirc(\mathcal A))
\le
\|\bcirc(\mathcal A)\|_F^2
=
n_3\|\mathcal A\|_F^2,
\]
so the contraction factor in \eqref{ineq:TRK} is nonnegative.
If \(\bcirc(\mathcal A)\) has full column rank, then
\(\sigma_{\min}(\bcirc(\mathcal A))>0\), so the contraction factor is strictly smaller than one. 
In this case, the expected error converges to zero as \(t\to\infty\) in the consistent noiseless setting.
This behavior is illustrated in Subsubsection~\ref{ex511}.

\subsection{Standard TRK under the doubly noisy model}
\label{sec:doubly_error}

We now formalize the doubly noisy tensor model described in the Introduction. Let
\(\mathcal{A}\) and \(\mathcal{B}\) be the noiseless tensors satisfying
\begin{equation}\label{eq:t_linearsystem}
\mathcal{A}  \mathcal{X}^\star = \mathcal{B},
\end{equation}
so that the underlying tensor linear system is consistent. The observed tensors are
\begin{equation}\label{eq:doubly}
\tilde{\mathcal{A}} = \mathcal{A} + \mathcal E_A,
\qquad
\tilde{\mathcal{B}} = \mathcal{B} + \mathcal E_B,
\end{equation}
where \(\mathcal E_A\) and \(\mathcal E_B\) denote the perturbation tensors. The
corresponding observed system is
\begin{equation}\label{eq:doublytls}
    \tilde{\mathcal{A}}  \mathcal{X} = \tilde{\mathcal{B}}.
\end{equation}
This system is not necessarily consistent and need not have
\(\mathcal X^\star\) as an exact solution.

To analyze standard TRK under \eqref{eq:doubly}, we study the error tensor
\[
\mathcal{E}^t := \mathcal{X}^t - \mathcal{X}^\star.
\]
This lets us separate projection-induced contraction from noise injection at each
iteration. Throughout the analysis below, \(\mathbb{E}[\cdot]\) denotes expectation
with respect to the randomized row selection. Under
Assumption~\ref{assump:invertibility} for the observed tensor
\(\tilde{\mathcal A}\), the projection operator is well-defined. The error update
then admits the following orthogonal decomposition.

\begin{lemma} \label{lem:orthogonality}
Assume that Assumption~\ref{assump:invertibility} holds for the observed tensor $\tilde{\mathcal{A}}$.
Let $\tilde{\mathcal{P}}_{i_t} = \tilde{\mathcal{A}}_{i_t::}^* (\tilde{\mathcal{A}}_{i_t::}\tilde{\mathcal{A}}_{i_t::}^*)^{-1} \tilde{\mathcal{A}}_{i_t::}$ be the orthogonal projection operator onto the tensor subspace associated with the $i_t$-th row slice of the observed tensor $\tilde{\mathcal{A}}$. The error update of the TRK algorithm satisfies the following orthogonal decomposition:
\begin{equation*}
\begin{aligned}
\|\mathcal{E}^{t+1}\|_F^2
&=
\|(\mathcal{I} - \tilde{\mathcal{P}}_{i_t})\mathcal{E}^t\|_F^2  \\
&\quad+
\left\|
\tilde{\mathcal{A}}_{i_t::}^*
(\tilde{\mathcal{A}}_{i_t::}\tilde{\mathcal{A}}_{i_t::}^*)^{-1}
\left(
(\mathcal E_A)_{i_t::}\mathcal X^\star
-
(\mathcal E_B)_{i_t::}
\right)
\right\|_F^2 .
\end{aligned}
\end{equation*}
\end{lemma}

Lemma~\ref{lem:orthogonality} is proved in Appendix~\ref{sec:orthogonality}.
Lemma~\ref{lem:orthogonality} reveals the one-step structure of the error in the doubly noisy setting. 
The updated error decomposes orthogonally into two components. 
One is the contracted previous error, and the other is a newly injected noise term. 
As a result, the squared error is the exact sum of these two contributions. This decomposition identifies the basic mechanism by which persistent perturbations can limit the attainable accuracy of standard TRK.

Building upon the orthogonal decomposition established in Lemma~\ref{lem:orthogonality}, our next goal is to quantify the expected one-step behavior of the error. By taking the expectation over the randomized row selection, the following theorem provides an upper bound on the expected error at the next iteration. The resulting recursion reflects the competition between the contraction of the previous error and the persistent perturbation term, while also accounting for the frequency imbalance of the observed row slices.

\begin{theorem} \label{thm:upper_recursion}
Let $\tilde{\mathcal{P}}_{i_t} = \tilde{\mathcal{A}}_{i_t::}^* (\tilde{\mathcal{A}}_{i_t::}\tilde{\mathcal{A}}_{i_t::}^*)^{-1} \tilde{\mathcal{A}}_{i_t::}$ be the orthogonal projection operator onto the tensor subspace associated with the $i_t$-th row slice of the observed tensor $\tilde{\mathcal{A}}$. 
Assume that Assumption~\ref{assump:invertibility} holds for the observed tensor $\tilde{\mathcal{A}}$, and that the block-circulant matrix $\bcirc(\tilde{\mathcal{A}})$ has full column rank.  
Let
\begin{equation}\label{eq:posindex}
\mathcal I_+
:=
\left\{
i\in[n_1]:\|\tilde{\mathcal A}_{i::}\|_F>0
\right\}.
\end{equation}
For each \(i\in\mathcal I_+\), define
\[
\gamma_i
=
\frac{\max\limits_{1 \le \ell \le n_3} \|\hat{\tilde{A}}_{i:}^{(\ell)}\|_F^2}
{\min\limits_{1 \le \ell \le n_3} \|\hat{\tilde{A}}_{i:}^{(\ell)}\|_F^2},
\qquad
\gamma = \max_{i\in\mathcal I_+} \gamma_i.
\]
Here, we use the Fourier-domain notation introduced in Section~\ref{sec:pre}.
Since Assumption~\ref{assump:invertibility} ensures that every row slice sampled with positive probability has nonzero Fourier row norm in every frequency, \(\gamma_i\) is well-defined, and \(\gamma \ge 1\) measures the maximum frequency imbalance across the Fourier slices of the observed row tensors that can be sampled with positive probability.
If $i_t$ is sampled independently at each iteration $t$ according to the distribution $p_i = \|\tilde{\mathcal{A}}_{i::}\|_F^2 / \|\tilde{\mathcal{A}}\|_F^2$ for $i \in [n_1]$, then the expected error of the TRK algorithm under the doubly noisy model satisfies
\begin{equation*}
\mathbb{E}\left[ \|\mathcal{E}^{t+1}\|_F^2 \right]
\le
\left(
1-
\frac{\sigma_{\min}^2(\bcirc(\tilde{\mathcal{A}}))}
{n_3 \|\tilde{\mathcal{A}}\|_F^2}
\right)
\mathbb{E}\left[ \|\mathcal{E}^t\|_F^2 \right]
+
\frac{\gamma \|\mathcal E_A \mathcal{X}^\star - \mathcal E_B\|_F^2}{\|\tilde{\mathcal{A}}\|_F^2}.
\end{equation*}
\end{theorem}

Theorem~\ref{thm:upper_recursion} is proved in Appendix~\ref{sec:upper_recursion}.
This result shows that the expected TRK error is no longer governed by a purely contractive recursion under the doubly noisy model. 
Instead, each iteration combines a contraction term controlled by the tensor scaling
\(n_3\|\tilde{\mathcal A}\|_F^2\) and a persistent additive perturbation term, which is
amplified by the frequency imbalance constant \(\gamma\). In the noiseless limit
\(\mathcal E_A=\mathcal E_B=0\), the recursion reduces exactly to the standard TRK rate in
Theorem~\ref{thm:standard_convergence}. This observation motivates introducing the
spectrally damped TRK method.

\section{Spectrally Damped TRK (SD-TRK)}\label{sec:sdtrk}

In Section~\ref{sec:error}, we showed that standard TRK~\eqref{eq:TRK} does not necessarily converge to the exact solution under the doubly noisy model. 
To address this limitation, we introduce a spectrally damped TRK method for
doubly noisy tensor systems. The analysis below is stated for fixed perturbation
tensors, while the numerical experiments use Gaussian perturbations.

To address the instability inherent in the doubly noisy setting, we incorporate both a relaxation parameter and spectral damping to enhance stability.
A relaxation parameter controls the step size of each update. In noisy or inconsistent settings, smaller values are often used to obtain more stable iterations \cite{Censor1983,Nikazad2024}.

Furthermore, perturbations in the system tensor $\tilde{\mathcal{A}}$ can make the row-slice normalization term $\tilde{\mathcal{A}}_{i_t::}\tilde{\mathcal{A}}_{i_t::}^{*}$ in the standard TRK update~\eqref{eq:TRK} poorly scaled.
In particular, small row-slice normalization values can lead to large reciprocal factors in its inverse, which may amplify perturbations in the update~\eqref{eq:TRK}.
Therefore, we apply a Tikhonov-type regularization to the inverse normalization term \cite{Hansen1998,Tikhonov1963}. 
This replaces the inverse normalization factor with
\begin{equation}\label{eq:tikhonov}
\left(
\tilde{\mathcal{A}}_{i_t::}\tilde{\mathcal{A}}_{i_t::}^{*}
+
\lambda\mathcal{I}
\right)^{-1},
\end{equation}
where $\lambda\ge 0$ is the damping parameter.
This regularized inverse damps the effect of small normalization values rather than inverting them without regularization.
In the Fourier domain, this regularization can be interpreted as a form of spectral damping, since it reduces the influence of ill-conditioned frequency components \cite{Chung2011,Hansen2006}.

This viewpoint also motivates the slice-dependent implementation used later:
different Fourier slices can have different conditioning, so a single global damping
level may not reflect the local scale of every spectral component.

By incorporating a relaxation parameter and spectral damping into the standard TRK iteration~\eqref{eq:TRK}, we propose
the Spectrally Damped TRK (SD-TRK) update as follows.
\begin{equation}\label{eq:SDTRK}
\mathcal{X}^{t+1}
=
\mathcal{X}^t
-
\omega \tilde{\mathcal{A}}_{i_t::}^*
\left(
\tilde{\mathcal{A}}_{i_t::}\tilde{\mathcal{A}}_{i_t::}^*
+
\lambda \mathcal{I}
\right)^{-1}
\left(
\tilde{\mathcal{A}}_{i_t::}\mathcal{X}^t
-
\tilde{\mathcal{B}}_{i_t::}
\right),
\end{equation}
where $\omega \in (0,2)$ is the relaxation parameter controlling the step size and $\lambda \ge 0$ is the damping parameter that regularizes the inverse normalization term.
When $\omega = 1$ and $\lambda = 0$, the update reduces to the standard TRK iteration \eqref{eq:TRK}.
Thus, SD-TRK can be viewed as a generalization of the standard TRK method.

The role of the damping parameter $\lambda$ is more explicit in the implementation. 
Since the $t$-product is evaluated by applying the FFT along the third dimension, the SD-TRK update \eqref{eq:SDTRK} is computed slice by slice in the Fourier domain. 
Recall that $\hat{\tilde A}^{(k)}$ denotes the $k$-th Fourier frontal slice of $\tilde{\mathcal{A}}$, and $\hat{\tilde A}_{i_t:}^{(k)}$ denotes its $i_t$-th row. 
For this row, the inverse factor in the $k$-th slice is
\[
\left(
\hat{\tilde A}_{i_t:}^{(k)}
\hat{\tilde A}_{i_t:}^{(k)*}
+
\lambda I
\right)^{-1}.
\]
The addition of $\lambda I$ prevents small row-slice energies \(\|\hat{\tilde A}_{i_t:}^{(k)}\|_F^2\) from producing large reciprocal factors.
Thus, the damping step is analogous to Tikhonov-type regularization within the $t$-product framework \cite{KilmerMartin2011,Tikhonov1963}, and it makes the row update less sensitive to weakly scaled Fourier row slices.

\subsection{Error analysis}\label{subsec:sdtrkerror}

We now derive the error analysis for SD-TRK under the doubly noisy model \eqref{eq:doublytls}. 
Recall that the error tensor is defined by
\[
\mathcal{E}^t := \mathcal{X}^t - \mathcal{X}^\star,
\]
where $\mathcal{X}^\star$ is the exact solution of the noiseless system \eqref{eq:t_linearsystem}.

\begin{lemma}
\label{lem:sdtrk_error_update}
If $\lambda=0$, assume that Assumption~\ref{assump:invertibility} holds for the observed tensor $\tilde{\mathcal{A}}$.
Let
\begin{equation}\label{eq:Qit}
\tilde{\mathcal{Q}}_{i_t}^{\lambda}
:=
\tilde{\mathcal{A}}_{i_t::}^*
\left(
\tilde{\mathcal{A}}_{i_t::}\tilde{\mathcal{A}}_{i_t::}^*
+
\lambda \mathcal{I}
\right)^{-1}
\tilde{\mathcal{A}}_{i_t::},
\end{equation}
where $i_t \in [n_1]$ is the index of the horizontal slice selected at iteration $t$,
and $\lambda \ge 0$ is the damping parameter.
Then the SD-TRK error satisfies
\[
\mathcal{E}^{t+1}
=
(\mathcal{I} - \omega \tilde{\mathcal{Q}}_{i_t}^{\lambda})\mathcal{E}^t
-
\omega
\tilde{\mathcal{A}}_{i_t::}^*
\left(
\tilde{\mathcal{A}}_{i_t::}\tilde{\mathcal{A}}_{i_t::}^*
+
\lambda \mathcal{I}
\right)^{-1}
\mathcal{N}_{i_t},
\]
where
\[
\mathcal{N}_{i_t}
=
(\mathcal E_A)_{i_t::} \mathcal{X}^\star - (\mathcal E_B)_{i_t::}.
\]
\end{lemma}

\begin{proof}
By substituting $\mathcal{X}^t = \mathcal{E}^t + \mathcal{X}^\star$ into the residual and using $\mathcal{A}_{i_t::}\mathcal{X}^\star - \mathcal{B}_{i_t::} = 0$, the same computation as in the proof of Lemma~\ref{lem:orthogonality} gives
\[
\tilde{\mathcal{A}}_{i_t::}\mathcal{X}^t - \tilde{\mathcal{B}}_{i_t::}
=
\tilde{\mathcal{A}}_{i_t::}\mathcal{E}^t + \mathcal{N}_{i_t},
\qquad
\mathcal{N}_{i_t}
=
(\mathcal E_A)_{i_t::} \mathcal{X}^\star - (\mathcal E_B)_{i_t::}.
\]
Substituting this into the SD-TRK update rule \eqref{eq:SDTRK} and using $\mathcal{X}^t = \mathcal{E}^t + \mathcal{X}^\star$ again, we obtain
\[
\mathcal{E}^{t+1}
=
\mathcal{E}^t
-
\omega \tilde{\mathcal{A}}_{i_t::}^*
\left(
\tilde{\mathcal{A}}_{i_t::}\tilde{\mathcal{A}}_{i_t::}^* + \lambda \mathcal{I}
\right)^{-1}
\left(
\tilde{\mathcal{A}}_{i_t::}\mathcal{E}^t + \mathcal{N}_{i_t}
\right),
\]
which is the claimed identity.
\end{proof}

Lemma~\ref{lem:sdtrk_error_update} shows that the SD-TRK error consists of a regularized error reduction term and a noise-dependent perturbation term. 
When \(\lambda>0\), unlike standard TRK \eqref{eq:TRK}, the operator \(\tilde{\mathcal{Q}}_{i_t}^{\lambda}\) is no longer an exact orthogonal projector, since the regularization modifies the projection step. 
By replacing
\[
\left(\tilde{\mathcal{A}}_{i_t::}\tilde{\mathcal{A}}_{i_t::}^*\right)^{-1}
\quad\text{with}\quad
\left(\tilde{\mathcal{A}}_{i_t::}\tilde{\mathcal{A}}_{i_t::}^*+\lambda\mathcal{I}\right)^{-1},
\]
the update \eqref{eq:SDTRK} becomes less sensitive to weakly scaled row-slice normalization values associated with the selected row slice $\tilde{\mathcal{A}}_{i_t::}$. 
We next derive an expected recursion that captures the trade-off between error reduction and noise perturbation.

\begin{theorem}
\label{thm:sdtrk_recursion}
Assume that Assumption~\ref{assump:invertibility} holds for the observed tensor
\(\tilde{\mathcal A}\), and that the block-circulant matrix
\(\bcirc(\tilde{\mathcal A})\) has full column rank.
Let \(\omega\in(0,2)\), \(\lambda\ge0\), and suppose that \(i_t\) is sampled independently according to
\begin{equation}\label{eq:sampleprob}
p_i=\frac{\|\tilde{\mathcal A}_{i::}\|_F^2}{\|\tilde{\mathcal A}\|_F^2},
\qquad i\in[n_1].
\end{equation}
For any \(\alpha>0\), the expected SD-TRK error satisfies
\begin{equation}\label{eq:sdtrkerror}
\mathbb E\!\left[\|\mathcal E^{t+1}\|_F^2\right]
\le
\rho_\alpha(\omega,\lambda)\,
\mathbb E\!\left[\|\mathcal E^t\|_F^2\right]
+
\eta_\alpha(\omega,\lambda)\,
\|\mathcal E_A\mathcal X^\star-\mathcal E_B\|_F^2,
\end{equation}
where
\[
\underline p:=\min_{i\in[n_1]:\,p_i>0} p_i,
\]
\[
\rho_\alpha(\omega,\lambda)
:=
(1+\alpha)\left(
1-
\underline p\,
\frac{(2\omega-\omega^2)\sigma_{\min}^2(\bcirc(\tilde{\mathcal A}))}
{\sigma_{\max}^2(\bcirc(\tilde{\mathcal A}))+\lambda}
\right),
\]
and
\[
\eta_\alpha(\omega,\lambda)
:=
\left(1+\frac{1}{\alpha}\right)\eta(\omega,\lambda),
\]
where
\begin{equation}\label{eq:eta}
\eta(\omega,\lambda)
:=
n_3\,
\omega^2
\max_{i\in\mathcal I_+}
\left\|
\tilde{\mathcal A}_{i::}^*
(
\tilde{\mathcal A}_{i::}\tilde{\mathcal A}_{i::}^*
+
\lambda \mathcal I
)^{-1}
\right\|_F^2
\end{equation}
for $\mathcal I_+$ defined as in \eqref{eq:posindex}.
\end{theorem}

Theorem~\ref{thm:sdtrk_recursion} is proved in
Appendix~\ref{sec:sdtrk_recursion}. The recursion \eqref{eq:sdtrkerror} separates
the expected error into a contraction-type term and a persistent perturbation term.
If \(\rho_\alpha(\omega,\lambda)<1\), it gives contraction up to a noise-dependent
error floor. If this condition fails, the bound still describes the effect of
damping and relaxation, but it does not by itself imply decay or identify the
limiting point of the iterates.

The factor \(\rho_\alpha(\omega,\lambda)\) controls error propagation, whereas
\(\eta_\alpha(\omega,\lambda)\) controls noise injection. Underrelaxation and
stronger damping can reduce noise amplification at the cost of a weaker contraction
bound. Hence the theorem captures a speed--robustness trade-off. The parameter
\(\alpha\) comes from Young's inequality and is not part of the algorithm.

The dependence on \(\underline p\) is a worst-case feature of the proof. If the
row-sampling distribution is highly imbalanced, \(\underline p\) may be very small,
so the contraction factor can be pessimistic. The theorem should therefore be read
as identifying the contraction--perturbation structure and the roles of
\(\omega\) and \(\lambda\), rather than as a sharp predictor of the empirical
convergence rate.

Since the observed system \eqref{eq:doublytls} is generally inconsistent, we next
compare the iterates with the least-squares solution of the observed system.

We denote by
\begin{equation}\label{eq:LSsol}
\mathcal{X}^{LS}:=\tilde{\mathcal{A}}^\dagger \tilde{\mathcal{B}}
\end{equation}
the least-squares solution associated with the observed system \eqref{eq:doublytls},
where \(\tilde{\mathcal{A}}^\dagger\) denotes the \(t\)-pseudoinverse of
\(\tilde{\mathcal{A}}\) \cite{KilmerMartin2011}. Equivalently,
\[
\unfold(\mathcal X^{LS})
=
\bcirc(\tilde{\mathcal A})^\dagger
\unfold(\tilde{\mathcal B}).
\]

As a baseline motivated by the doubly noisy matrix Kaczmarz analysis in
\cite{BBDLM2024}, we introduce the following matrix-inspired reference scale for
interpreting the error relative to the least-squares solution:
\begin{equation}\label{eq:lsbound}
\mathcal H_{\mathrm{LS}}(t)
:=
\left(1-\frac{1}{\tilde{R}}\right)^t
\|\mathcal{X}^0-\mathcal{X}^{LS}\|_F^2
+
\frac{\|\mathcal E_A\mathcal{X}^{LS}-\mathcal E_B\|_F^2}
{\sigma_{\min}^2(\bcirc(\tilde{\mathcal{A}}))}.
\end{equation}
This quantity is used only as a comparison scale and should not be read as a convergence theorem for the slice-dependent Fourier-domain implementation in Algorithm~\ref{alg:sdtrk}. Here
\[
\tilde R
:=
\|\bcirc(\tilde{\mathcal A})^\dagger\|_2^2
\|\bcirc(\tilde{\mathcal A})\|_F^2
=
n_3
\|\bcirc(\tilde{\mathcal A})^\dagger\|_2^2
\|\tilde{\mathcal A}\|_F^2.
\]

The reference scale \eqref{eq:lsbound} illustrates how small singular values of \(\bcirc(\tilde{\mathcal A})\) can affect a noisy Kaczmarz-type iteration in two
ways. They can enlarge
\(\tilde R\), thereby slowing the decay of the transient term, and they can also
increase the perturbation-dependent error horizon through the factor
\(\sigma_{\min}^{-2}(\bcirc(\tilde{\mathcal A}))\). Thus, small singular values can
simultaneously slow convergence and amplify the asymptotic effect of noise.

The SD-TRK update addresses this mechanism directly at the Fourier-slice level.
In the \(k\)-th Fourier slice, the normalization factor is replaced by
\[
\left(
\hat{\tilde A}_{i_t:}^{(k)}
\hat{\tilde A}_{i_t:}^{(k)*}
+
\lambda^{(k)}I
\right)^{-1},
\]
so the damping parameter \(\lambda^{(k)}\) reduces the effect of directions
associated with small row-slice energies. This spectral damping is not uniformly
beneficial: increasing \(\lambda^{(k)}\) may improve robustness to noisy
ill-conditioned slices, but it can also slow the transient error decay. Thus,
SD-TRK introduces a trade-off between asymptotic robustness and convergence speed.

Theorem~\ref{thm:sdtrk_recursion} provides a scalar-parameter error recursion,
while the Fourier-slice discussion above gives an implementation-level
interpretation of the corresponding damping mechanism. In
Section~\ref{sec:experiments}, we empirically examine this stabilization behavior
through multi-trial synthetic experiments, a noise-scaling comparison, and an image reconstruction example.

\subsection{Implementation}

In this subsection, we present the practical implementation of SD-TRK. 
As discussed in Appendix~\ref{sec:appendix_fft}, the t-product is implemented in the Fourier domain by performing matrix multiplication on each frontal slice.
Accordingly, the SD-TRK iteration is also implemented slice-wise in the Fourier domain.
Here we follow the Fourier-domain notation introduced in Section~\ref{sec:pre} 
and Algorithm~\ref{alg:t_product_fft}.
Thus, even for real-valued input tensors, the Fourier-domain variables in
Algorithm~\ref{alg:sdtrk} are generally complex-valued, and \(^{*}\) denotes the
conjugate transpose in the slice-wise updates.

For the theoretical analysis in Subsection~\ref{subsec:sdtrkerror}, we use a
uniform scalar damping parameter \(\lambda\) in order to obtain a tractable global
recursion. The practical algorithm applies the same damped update mechanism
slice-wise in the Fourier domain and therefore allows a frequency-dependent family
of damping parameters \(\{\lambda^{(k)}\}_{k=1}^{n_3}\).

This slice-dependent scheme should be viewed as an implementation-level
extension motivated by the scalar-parameter theory.
It is not a direct corollary of
Theorem~\ref{thm:sdtrk_recursion}, whose formal recursion is stated for a single
scalar damping parameter. In practice, allowing \(\lambda^{(k)}\) to vary across
slices provides additional flexibility in stabilizing ill-conditioned Fourier
components.
In the synthetic tensor experiments in Subsubsections~\ref{ex513}--\ref{ex516},
the slice-dependent damping parameters are chosen as
\begin{equation}\label{eq:lambda_scale_rule}
\lambda^{(k)}
=
\lambda_{\mathrm{floor}}
+
\lambda_{\mathrm{scale}}
(\delta_A^2+\delta_B^2)
\frac{\|\hat{\tilde A}^{(k)}\|_F^2}{n_1},
\qquad k=1,\ldots,n_3,
\end{equation}
where \(\lambda_{\mathrm{floor}}>0\) prevents degeneracy and
\(\lambda_{\mathrm{scale}}>0\) controls the overall damping strength.

The scaling in \eqref{eq:lambda_scale_rule} is intended to match the damping level
to both the noise magnitude and the local Fourier-slice scale. The factor
\(\delta_A^2+\delta_B^2\) reflects the combined perturbation level in the operator
and right-hand side. The quantity
\(\|\hat{\tilde A}^{(k)}\|_F^2/n_1\) is the average row-slice energy in the
\(k\)-th Fourier slice, so the damping parameter has the same scale as the
slice-wise normalization terms used in the update. The parameter
\(\lambda_{\mathrm{scale}}\) is therefore a user-chosen global multiplier rather
than an automatically selected optimal value. A systematic sensitivity analysis and
adaptive selection rule are left for future work.

When the perturbation levels \(\delta_A\) and \(\delta_B\) are not known in
advance, they should be replaced by estimated noise scales or absorbed into the
global multiplier \(\lambda_{\mathrm{scale}}\). Developing reliable data-driven
estimators for these quantities is beyond the scope of this paper.

In applications, this parameter should therefore be treated as a conservative
stabilization knob rather than as an automatically tuned quantity. Larger values
increase damping of weak spectral components, while smaller values keep the update
closer to standard TRK.

\begin{algorithm}[ht]
\caption{Spectrally Damped Tensor Randomized Kaczmarz (SD-TRK)}
\label{alg:sdtrk}
\begin{algorithmic}[1]
\Require Observed noisy tensors $\tilde{\mathcal{A}} \in \mathbb{R}^{n_1 \times n_2 \times n_3}$, $\tilde{\mathcal{B}} \in \mathbb{R}^{n_1 \times n_4 \times n_3}$, initial iterate $\mathcal{X}^0 \in \mathbb{R}^{n_2 \times n_4 \times n_3}$, maximum iteration count $T$, damping parameters $\{\lambda^{(k)}\}_{k=1}^{n_3}$ with $\lambda^{(k)} \ge 0$, relaxation parameter $\omega \in (0,2)$
\Ensure Approximate solution $\mathcal{X}^T$
\State Compute Fourier transforms:
\[
\hat{\tilde{\mathcal A}} = \fft(\tilde{\mathcal{A}}, [\,], 3), \qquad
\hat{\tilde{\mathcal B}} = \fft(\tilde{\mathcal{B}}, [\,], 3), \qquad
\hat{\mathcal X}^{0} = \fft(\mathcal{X}^{0}, [\,], 3)
\]
\For{$t=0,\dots,T-1$}
    \State Sample $i_t \in [n_1]$ with probability
    \[
    p_{i_t}=\frac{\|\tilde{\mathcal{A}}_{i_t::}\|_F^2}{\sum_{j=1}^{n_1}\|\tilde{\mathcal{A}}_{j::}\|_F^2}
    \]
    \For{$k=1,\dots,n_3$}
        \State Update the $k$-th frequency slice:
        \[
        \hat X^{t+1,(k)}
        =
        \hat X^{t,(k)}
        -
        \omega \hat{\tilde A}_{i_t:}^{(k)*}
        \left(
        \hat{\tilde A}_{i_t:}^{(k)}\hat{\tilde A}_{i_t:}^{(k)*}
        +\lambda^{(k)}I
        \right)^{-1}
        \left(
        \hat{\tilde A}_{i_t:}^{(k)}\hat X^{t,(k)}-\hat{\tilde B}_{i_t:}^{(k)}
        \right)
        \]
    \EndFor
\EndFor
\State Compute inverse Fourier transform:
\[
\mathcal{X}^T=\ifft(\hat{\mathcal X}^{T}, [\,], 3)
\]
\State \Return $\mathcal{X}^T$
\end{algorithmic}
\end{algorithm}

This frequency-dependent implementation allows the amount of damping to vary
across spectral components according to their local conditioning and noise levels.
In the following section, we present numerical experiments illustrating the
effectiveness of this damping strategy under doubly noisy perturbations.

\section{Numerical Experiments}
\label{sec:experiments}

This section presents numerical experiments related to the theoretical analysis and SD-TRK method.
All numerical experiments were performed on a MacBook Pro equipped with an Apple M3 Pro chip (12-core CPU) and 18 GB of unified memory, operating on macOS. 
The experiments were implemented in Python 3.13.3 using Jupyter notebooks.

The experiments are organized to address the following questions. 
First, Subsubsections~\ref{ex511}--\ref{ex512} illustrate the behavior of standard TRK in noiseless and doubly noisy settings, showing the transition from linear convergence to semi-convergence. 
Second, Subsubsections~\ref{ex513}--\ref{ex514} compare standard TRK and SD-TRK under well-conditioned and ill-conditioned noisy tensor systems, illustrating the stabilizing behavior of the spectrally damped update. 
Third, Subsubsection~\ref{ex515} uses a controlled known-target setting to separate the effect of the update mechanism from stochastic sampling variability. 
Fourth, Subsubsection~\ref{ex516} compares the empirical stabilization behavior of SD-TRK with a damped reference horizon. 
Finally, Subsection~\ref{sec:image_application} uses a two-pass image reconstruction
setting as a comparison between standard TRK and SD-TRK.

Throughout the synthetic experiments, we report three related but distinct
quantities. The true reconstruction error
\[
\|\mathcal X^t-\mathcal X^\star\|_F
\]
measures accuracy relative to the noiseless target and is available only in
controlled synthetic experiments. The observed residual
\[
\|\tilde{\mathcal A}\mathcal X^t-\tilde{\mathcal B}\|_F
\]
is the quantity available in practical noisy systems. When a least-squares
benchmark is used, we also report
\[
\|\mathcal X^t-\mathcal X^{LS}\|_F,
\qquad
\unfold(\mathcal X^{LS})
=
\bcirc(\tilde{\mathcal A})^\dagger
\unfold(\tilde{\mathcal B}).
\]
These quantities serve different purposes: the true error evaluates recovery
accuracy, the residual evaluates consistency with the observed noisy system, and
the least-squares error compares the iterate with the observed-system
least-squares solution.

Table~\ref{tab:experiment_summary} summarizes the main experimental settings used in the synthetic tensor experiments. 
All synthetic tensor experiments use dimensions
\((n_1,n_2,n_3,n_4)=(20,10,5,5)\).
The slice-dependent damping parameters in Subsubsections~\ref{ex513}--\ref{ex516}
are computed using \eqref{eq:lambda_scale_rule} with
\(\lambda_{\mathrm{floor}}=10^{-6}\), unless otherwise specified.

\begin{table}[ht]
\centering
\caption{Summary of synthetic tensor experiment settings}
\label{tab:experiment_summary}
\small
\renewcommand{\arraystretch}{1.25}
\begin{tabularx}{\textwidth}{c|Y|Y|Y|Y}
\hline
Experiment & Purpose & Noise setting & SD-TRK parameters & Iterations, trials \\
\hline
\ref{ex511}
&
Noiseless convergence of standard TRK
&
none
&
--
&
\(1000\); \(1\) trial
\\
\hline
\ref{ex512}
&
Semi-convergence of standard TRK
&
\(\delta_A=\delta_B\in\{0,10^{-4},10^{-2}\}\)
&
--
&
\(1000\); \(1\) trial per noise level
\\
\hline
\ref{ex513}
&
Well-conditioned comparison
&
\(\delta_A=10^{-3}\)\newline
\(\delta_B=0.5\)
&
\(\omega=0.9\)\newline
\(\lambda_{\mathrm{scale}}=5.0\)
&
\(2000\); \(20\) trials
\\
\hline
\ref{ex514}
&
Ill-conditioned comparison
&
\(\delta_A=0.01\)\newline
\(\delta_B=0.5\)
&
\(\omega=0.7\)\newline
\(\lambda_{\mathrm{scale}}=100.0\)
&
\(2000\); \(20\) trials
\\
\hline
\ref{ex515}
&
Paired known-target comparison
&
\(\delta_A=0.01\)\newline
\(\delta_B=0.5\)
&
\(\omega=0.8\)\newline
\(\lambda_{\mathrm{scale}}=50.0\)
&
\(2000\); \(10\) paired trials
\\
\hline
\ref{ex516}
&
Noise-scaling comparison
&
\(\delta_A=\delta_B=\delta\)\newline
\(\delta\in\{0,\allowbreak 10^{-3},\allowbreak {5\cdot10^{-3}},\allowbreak 10^{-2},\allowbreak {5\cdot10^{-2}}\}\)
&
\(\omega=0.9\)\newline
\(\lambda_{\mathrm{scale}}=5.0\)
&
\(1500\); \(10\) trials
\\
\hline
\end{tabularx}
\end{table}

\subsection{Synthetic tensor experiments}\label{sec:illustrative}

Throughout this subsection, we consider the synthetic tensor linear system \eqref{eq:t_linearsystem}.

\subsubsection{Standard TRK in the Noiseless Setting} \label{ex511}
We begin with a noiseless consistent tensor linear system to illustrate the behavior of standard TRK. 
The tensors $\mathcal{A}$ and $\mathcal{X}^\star$ were generated independently from the standard Gaussian distribution, and $\mathcal{B}=\mathcal{A} \mathcal{X}^\star$.
The initial iterate was set to the zero tensor, $\mathcal{X}^0=\mathbf{0}$, and the TRK iteration was run for $1000$ steps. 
The error was measured by the Frobenius norm $\|\mathcal{X}^t-\mathcal{X}^\star\|_F$. Figure~\ref{fig:illustrative_trk}(a) shows this error on a semilogarithmic scale. 
The approximately linear decay confirms the expected linear convergence of standard TRK in the noiseless consistent setting, in agreement with Theorem~\ref{thm:standard_convergence}.

\subsubsection{Standard TRK in the Doubly Noisy Setting} \label{ex512}
We next consider the doubly noisy setting.
The tensors \(\mathcal A\) and \(\mathcal B\) were perturbed according to
\[
\tilde{\mathcal A}=\mathcal A+\delta_A \mathcal G_A,
\qquad
\tilde{\mathcal B}=\mathcal B+\delta_B \mathcal G_B,
\]
where the entries of \(\mathcal G_A\) and \(\mathcal G_B\) were sampled
independently from the standard Gaussian distribution. Equivalently, in the
notation of \eqref{eq:doubly}, \(\mathcal E_A=\delta_A\mathcal G_A\) and
\(\mathcal E_B=\delta_B\mathcal G_B\).
We considered isotropic noise levels satisfying
\[
\delta_A=\delta_B\in\{0,10^{-4},10^{-2}\},
\]
where \(\delta_A=\delta_B=0\) corresponds to the noiseless reference setting.
The same noiseless tensors \(\mathcal A\) and \(\mathcal X^\star\) are used for all noise levels, and one run is performed per noise level.
We then applied standard TRK to \(\tilde{\mathcal A}\mathcal X=\tilde{\mathcal B}\), again starting from the zero tensor, and measured the error by \(\|\mathcal X^t- \mathcal{X}^\star\|_F\).
Figure~\ref{fig:illustrative_trk}(b) shows the corresponding error trajectories on a semilogarithmic scale.
While the noiseless case continues to decrease steadily, the noisy cases exhibit an initial decay followed by stabilization at nonzero error levels.
This semi-convergence behavior indicates that standard TRK is sensitive to doubly noisy perturbations and no longer converges to the exact solution in this setting.
The observed instability further motivates the introduction of stabilization mechanisms such as the proposed SD-TRK framework.

\begin{figure}[ht]
    \centering
    \begin{subfigure}{0.48\linewidth}
        \centering
        \includegraphics[width=\linewidth]{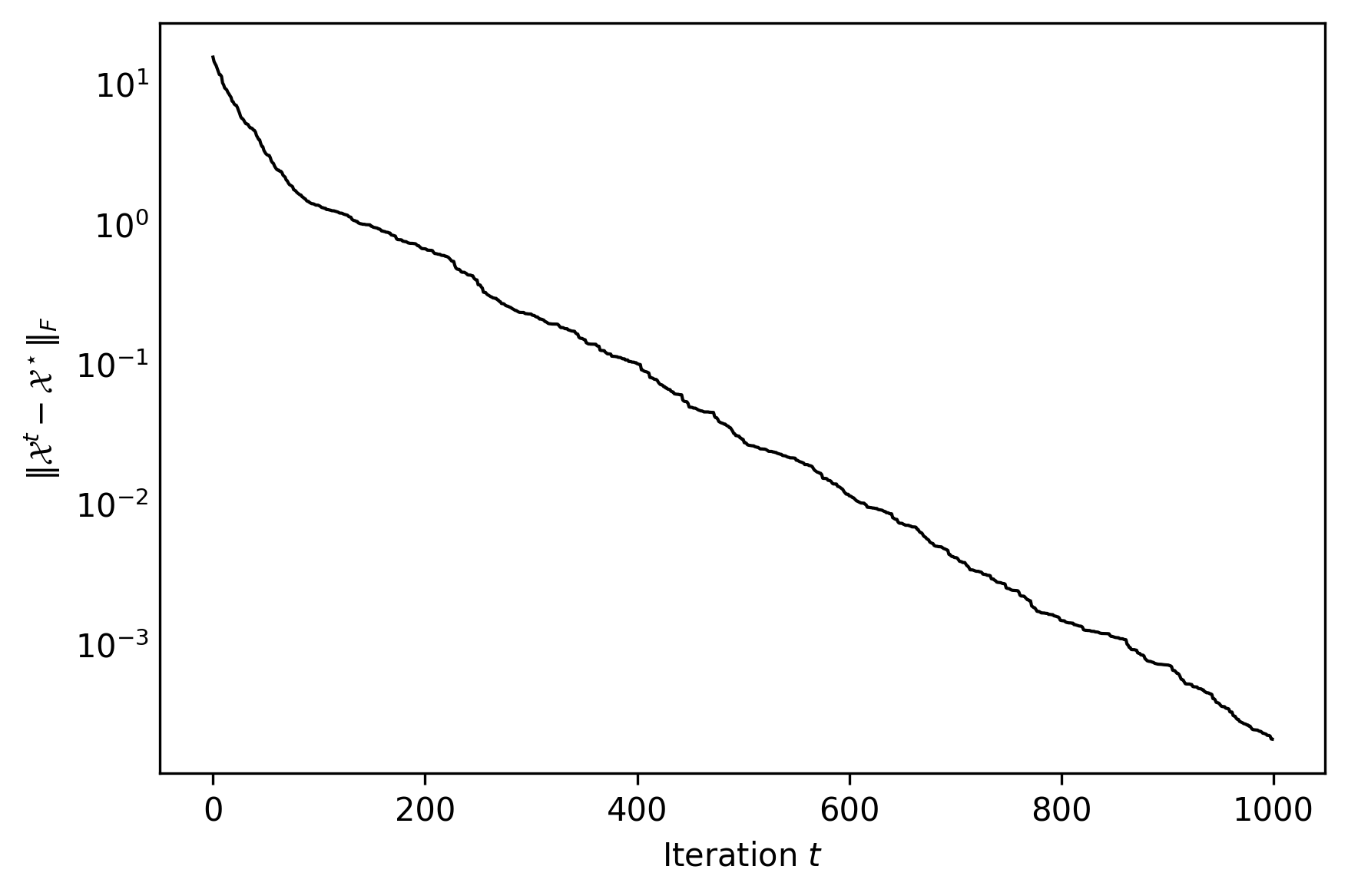}
        \caption*{(a) Noiseless setting}
    \end{subfigure}
    \hfill
    \begin{subfigure}{0.48\linewidth}
        \centering
        \includegraphics[width=\linewidth]{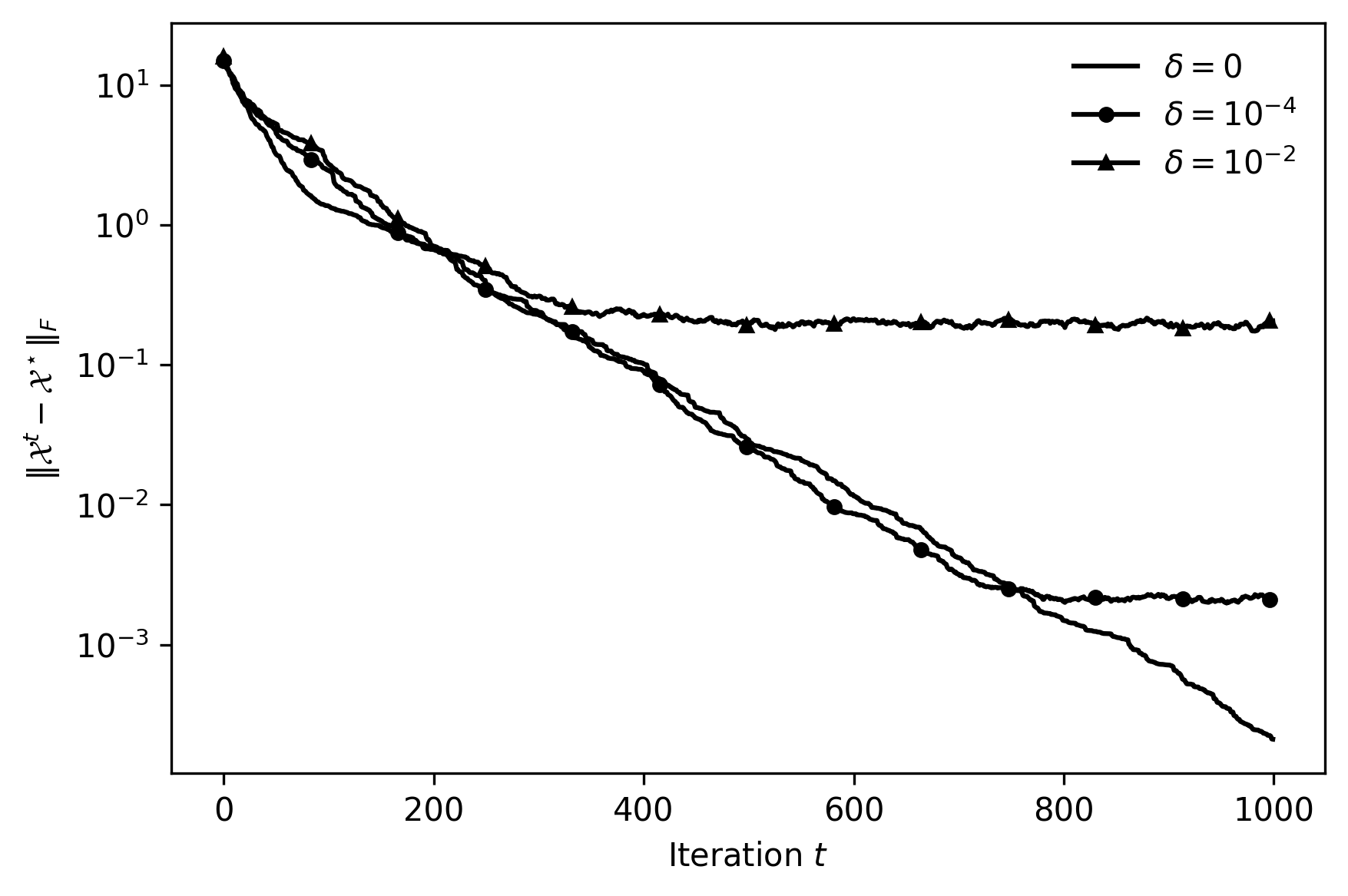}
        \caption*{(b) Doubly noisy setting}
    \end{subfigure}
    \caption{Convergence of standard TRK}
    \label{fig:illustrative_trk}
\end{figure}

\subsubsection{Well-Conditioned Doubly Noisy Comparison} \label{ex513}
We now consider a comparative experiment under a well-conditioned doubly noisy
tensor system. To isolate the influence of data-driven perturbations, we adopt an
asymmetric noise configuration with \(\delta_A = 10^{-3}\) and \(\delta_B = 0.5\).
This asymmetric choice is used to emphasize the effect of right-hand-side
perturbations while keeping structural distortion of the system tensor relatively
mild. In this way, the experiment focuses on the data-noise filtering behavior of
the proposed algorithm independently of severe operator perturbations.

The noiseless tensors \(\mathcal A\) and \(\mathcal X^\star\) are independently
generated from the standard Gaussian distribution, and the Gaussian noise tensors
\(\mathcal G_A\) and \(\mathcal G_B\) are generated independently. We then set
\[
\tilde{\mathcal A}
=
\mathcal A+\delta_A\mathcal G_A,
\qquad
\tilde{\mathcal B}
=
\mathcal A\mathcal X^\star+\delta_B\mathcal G_B.
\]
Equivalently, in the notation of \eqref{eq:doubly},
\(\mathcal E_A=\delta_A\mathcal G_A\) and
\(\mathcal E_B=\delta_B\mathcal G_B\).
Both algorithms are initialized from \(\mathcal X^0=\mathbf 0\) and executed for
\(T=2000\) iterations. The SD-TRK method employs \(\omega=0.9\) and uses
\eqref{eq:lambda_scale_rule} with \(\lambda_{\mathrm{scale}}=5.0\) and
\(\lambda_{\mathrm{floor}}=10^{-6}\).
All reported curves are averaged over \(20\) independent Monte Carlo trials.

Figure~\ref{fig:trk_multi}(a) and (b) present the averaged convergence trajectories of the standard TRK and SD-TRK methods, respectively.

\begin{figure}[ht]
    \centering
    \begin{subfigure}{0.48\linewidth}
        \centering
        \includegraphics[width=\linewidth]{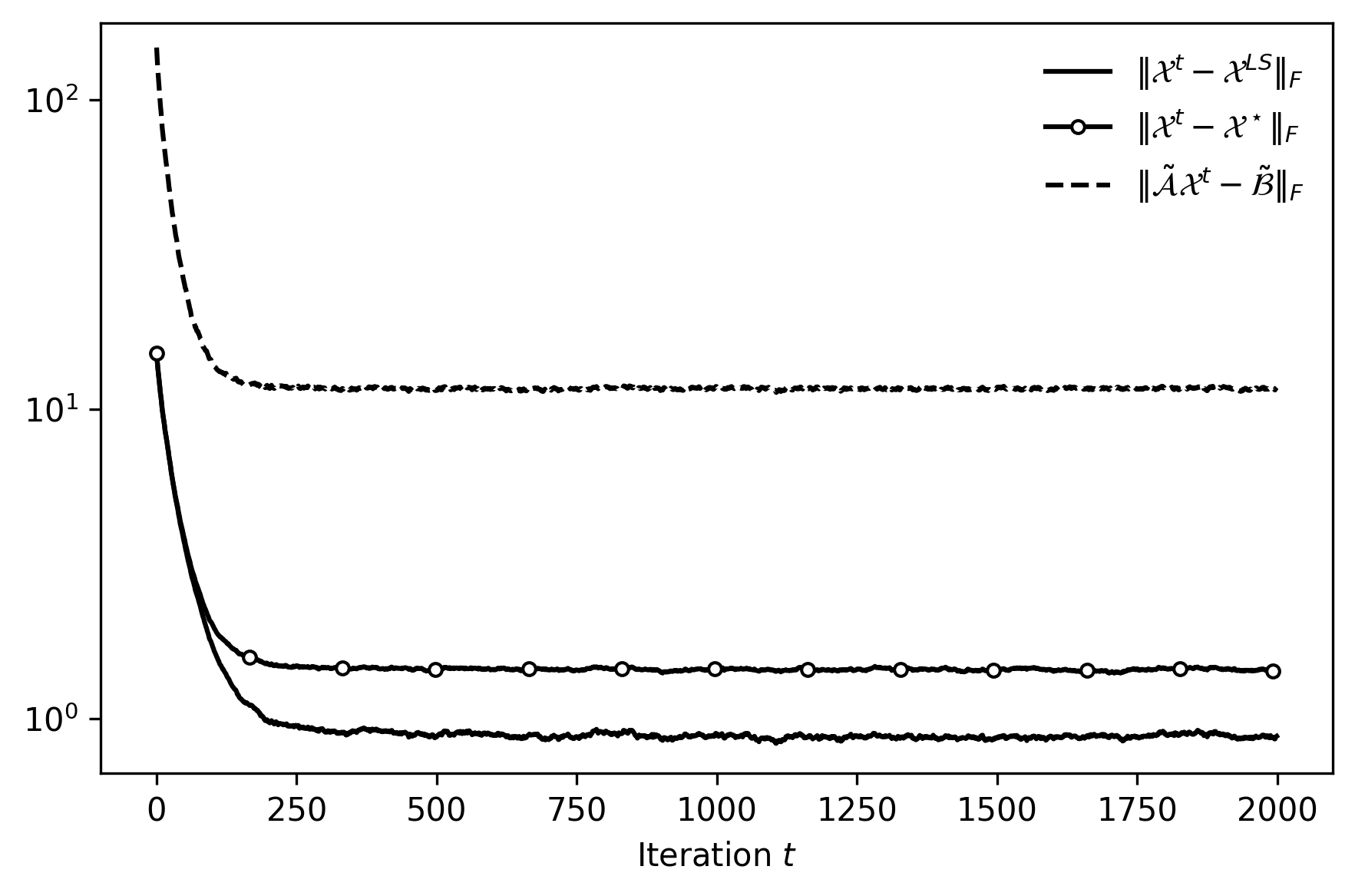}
        \caption*{(a) Well-Conditioned Standard TRK}
    \end{subfigure}
    \hfill
    \begin{subfigure}{0.48\linewidth}
        \centering
        \includegraphics[width=\linewidth]{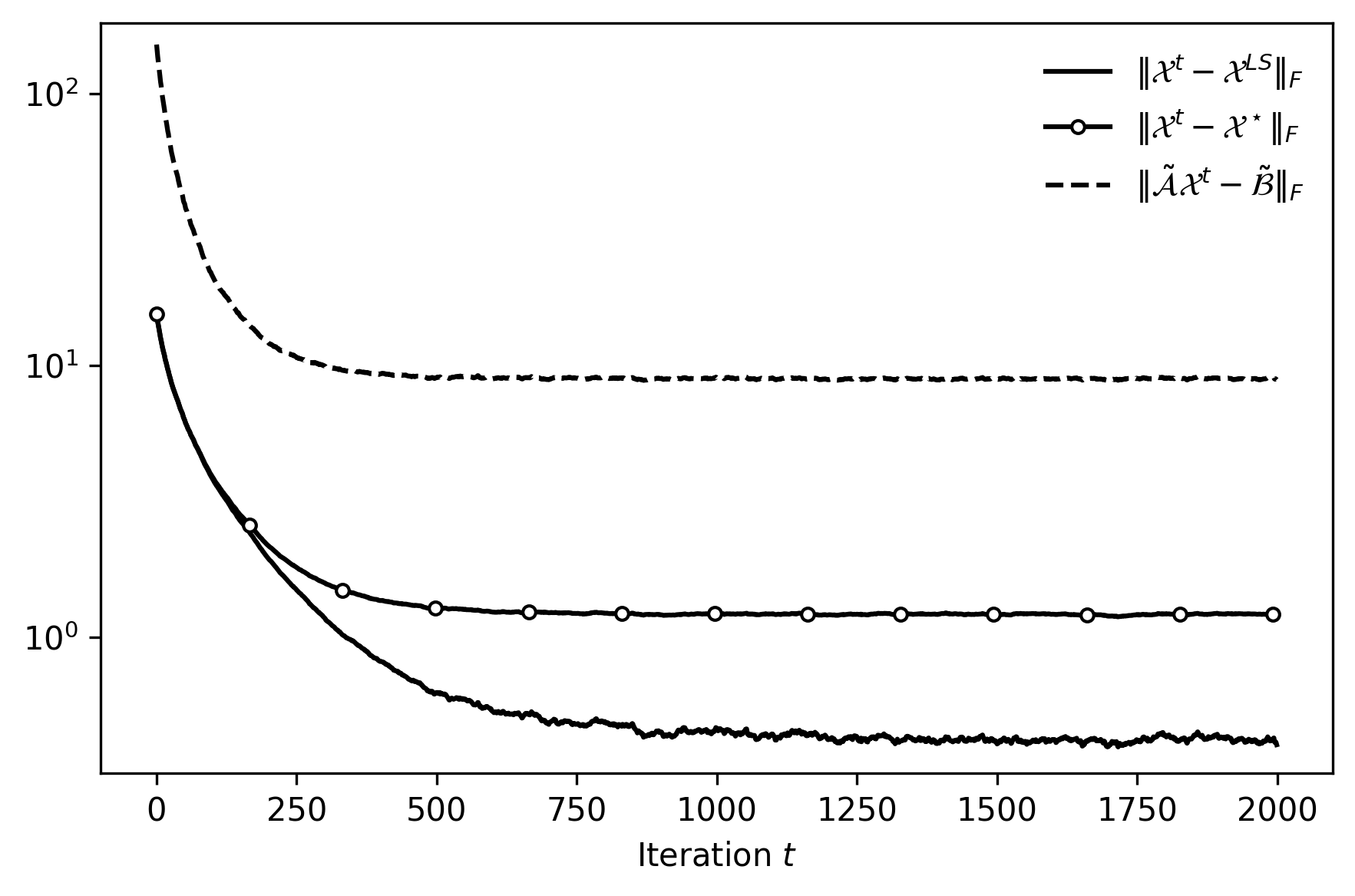}
        \caption*{(b) Well-Conditioned SD-TRK}
    \end{subfigure}
    \caption{Comparison of standard TRK and SD-TRK under a well-conditioned doubly noisy setting.}
    \label{fig:trk_multi}
\end{figure}

Under this well-conditioned setting, both methods exhibit stable iterative behavior. 
The standard TRK iteration attains a final least-squares error of \(0.8755\), a true reconstruction error \(\|\mathcal X^t- \mathcal{X}^\star\|_F\) of \(1.4214\), and a residual of \(11.7462\). In comparison, SD-TRK yields lower values for all three metrics, with a final least-squares error of \(0.4017\), a true reconstruction error of \(1.2020\), and a residual of \(8.8521\). Although the observation noise is substantial, the spectral slices of the system tensor remain reasonably well-conditioned, preventing strong amplification of the noise by ill-conditioned inverse steps.

Although the quantitative improvement is relatively modest in this well-conditioned regime, SD-TRK consistently achieves lower reconstruction error and residual values than standard TRK. Since the spectral slices do not exhibit severe ill-conditioning, the standard TRK method does not experience strong noise amplification and is therefore able to approximate the least-squares benchmark with reasonable accuracy.

A more significant distinction emerges in the relationship between the observable residual and the true reconstruction error. 
In practical applications, the ground-truth tensor \( \mathcal{X}^\star\) is unavailable, and the residual \(\|\tilde{\mathcal A}\mathcal X^t-\tilde{\mathcal B}\|_F\) becomes the primary quantity available for monitoring convergence.

For the standard TRK iteration shown in Figure~\ref{fig:trk_multi}(a), the observable residual rapidly decreases and reaches a nearly flat plateau after approximately \(150\)--\(200\) iterations. Since the residual is the only practically observable quantity in real-world applications, such behavior may incorrectly suggest that the reconstruction process has already converged. 
However, the true reconstruction error continues to decrease after the residual has nearly plateaued and remains noticeably separated from the residual trajectory.
As a result, residual-based stopping criteria may be difficult to calibrate reliably in this noisy setting.

In contrast, SD-TRK exhibits closer alignment between the residual trajectory and the stabilization of the true reconstruction error in Figure~\ref{fig:trk_multi}(b). Both quantities approach their plateau regions at approximately the same iteration stage. In this experiment, the observable residual is therefore a more informative diagnostic for SD-TRK than for standard TRK.

\subsubsection{Ill-Conditioned Doubly Noisy Comparison} \label{ex514}
We next examine the behavior of SD-TRK under a severely ill-conditioned doubly noisy tensor system. This experiment examines the stability of the reconstruction process when the underlying tensor operator contains highly ill-conditioned spectral components, a setting in which inverse instability and noise amplification become particularly pronounced.

To construct such a regime, we generate a Gaussian tensor and then multiply its \(j\)-th lateral slice by a scale factor \(c_j\), where the factors \(c_1,\ldots,c_{n_2}\) decay geometrically from \(1\) to \(2\times 10^{-2}\). 
This construction makes the smallest singular values of the Fourier-domain frontal slices \(\hat{\tilde A}^{(k)}\) small, thereby increasing the condition number of the tensor system. 
The ground-truth tensor \(\mathcal X^\star\) is initialized as the zero tensor. Its first four first-mode slices are filled with independent standard Gaussian entries scaled by \(0.75\), and all remaining first-mode slices are kept equal to zero.

To isolate the instability caused by the inversion of near-singular spectral components, we consider a minor numerical perturbation (\(\delta_A = 0.01\)) to the operator tensor, while a substantial observation noise (\(\delta_B = 0.5\)) is injected into the right-hand side tensor. All remaining dimensional settings and iteration counts are kept identical to those used in Subsubsection~\ref{ex513} in order to preserve consistency across the experiments. For SD-TRK, a stronger damping configuration is adopted with relaxation parameter \(\omega = 0.7\) together with an increased damping scale \(\lambda_{\mathrm{scale}} = 100.0\).
The corresponding slice-dependent damping parameters are computed from
\eqref{eq:lambda_scale_rule} with \(\lambda_{\mathrm{floor}}=10^{-6}\).

Figure~\ref{fig:ill_cond_trk} presents the empirical convergence trajectories of both methods, comparing the true reconstruction error \(\|\mathcal X^t- \mathcal{X}^\star\|_F\) together with the observable residual \(\|\tilde{\mathcal A}\mathcal X^t-\tilde{\mathcal B}\|_F\).

\begin{figure}[ht]
    \centering
    \begin{subfigure}{0.48\linewidth}
        \centering
        \includegraphics[width=\linewidth]{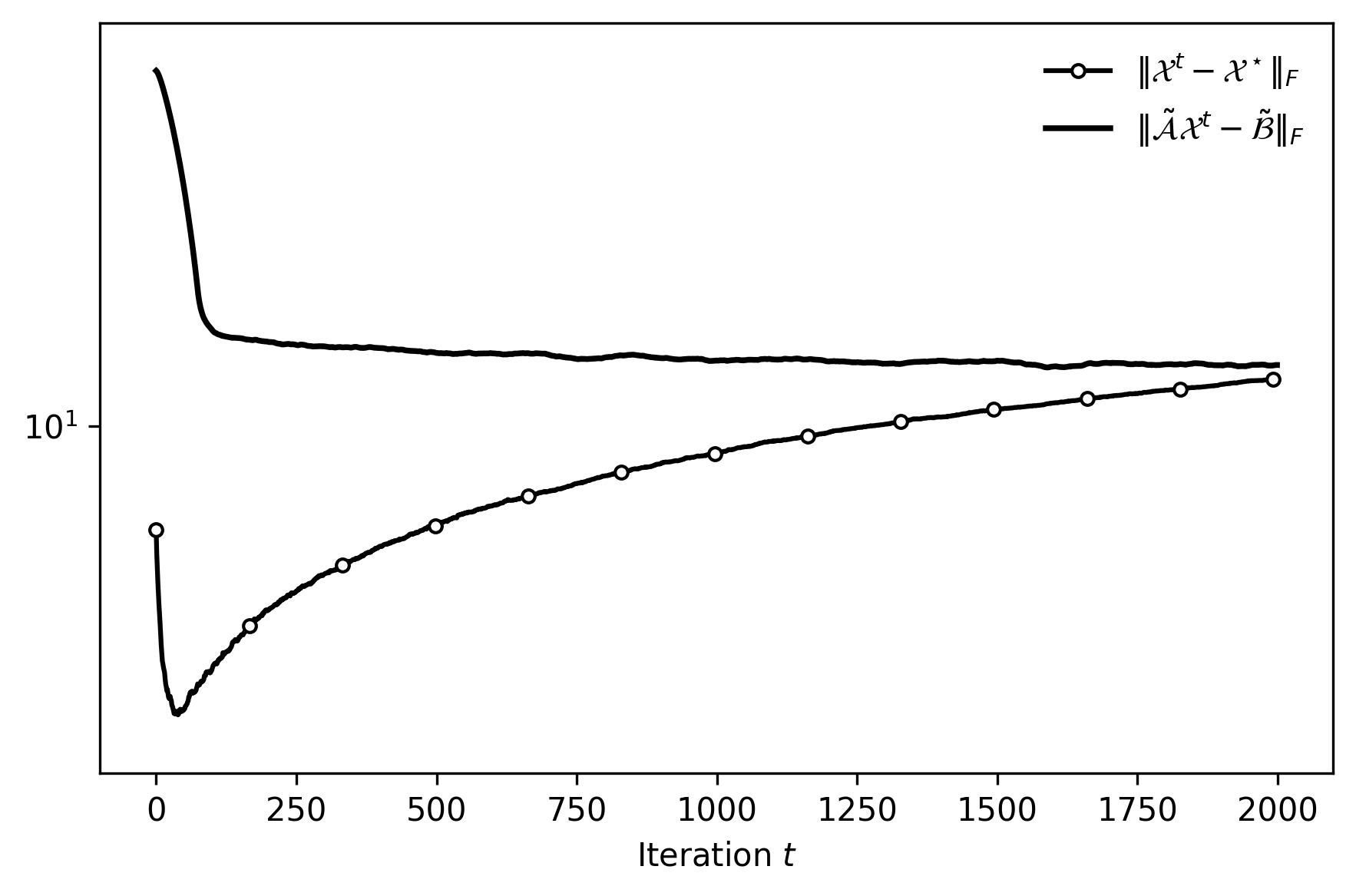}
        \caption*{(a) Standard TRK}
    \end{subfigure}
    \hfill
    \begin{subfigure}{0.48\linewidth}
        \centering
        \includegraphics[width=\linewidth]{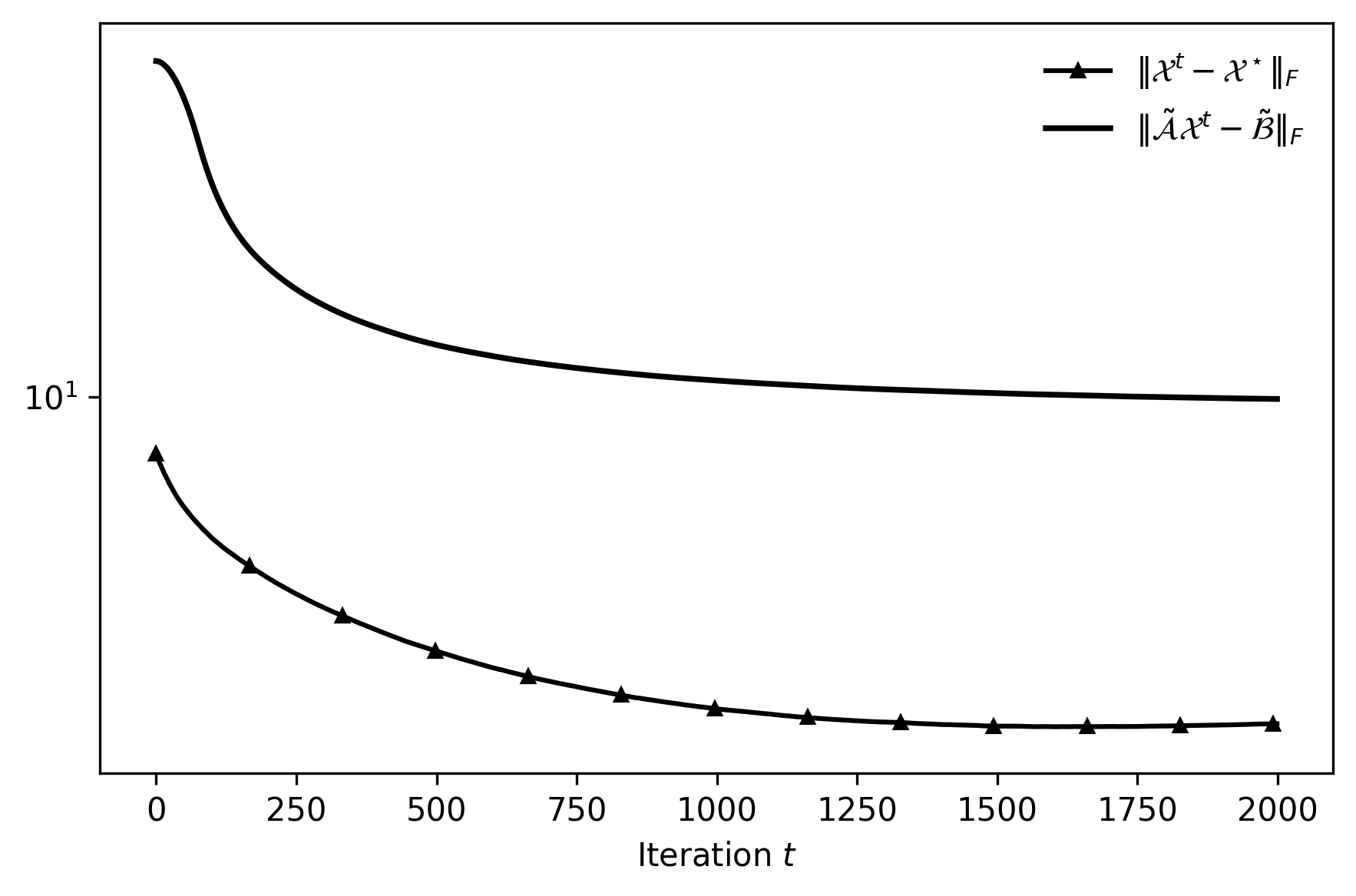}
        \caption*{(b) SD-TRK}
    \end{subfigure}
    \caption{Comparison of standard TRK and SD-TRK under an ill-conditioned doubly noisy setting.}
    \label{fig:ill_cond_trk}
\end{figure}

The empirical results reveal severe instability in the standard TRK framework under this ill-conditioned regime. As shown in Figure~\ref{fig:ill_cond_trk}(a), the true reconstruction error of standard TRK initially decreases during the early iterations, but subsequently reverses direction and eventually increases toward a large asymptotic error level of approximately \(12.1125\). This characteristic V-shaped trajectory illustrates the classical phenomenon of semi-convergence. Because the standard TRK iteration does not incorporate any regularization mechanism, it attempts to satisfy the noisy observation tensor \(\tilde{\mathcal B}\) through repeated inversion of severely ill-conditioned spectral components. In the presence of near-zero singular values, this process leads to substantial amplification of observation noise and significantly degrades the reconstruction accuracy.

In contrast, SD-TRK is empirically more stable under the same ill-conditioned
setting. The frequency-dependent damping parameters \(\{\lambda^{(k)}\}_{k=1}^{n_3}\) act as a spectral regularization mechanism, damping the amplification associated with near-zero singular values across the Fourier-domain frontal slices. Consequently, as illustrated in Figure~\ref{fig:ill_cond_trk}(b), the true reconstruction error of SD-TRK decreases nearly monotonically and eventually stabilizes near an error floor of approximately \(2.1013\). Thus, the final reconstruction error is approximately \(5.76\) times smaller than that of standard TRK.

The residual behavior also favors SD-TRK in this experiment. Even under severe ill-conditioning, the residual trajectory remains more closely aligned with the stabilization phase of the true reconstruction error and gradually approaches a stable plateau after approximately \(500\) iterations. In this experiment, the observable residual is therefore more informative for SD-TRK than for standard TRK. These results are consistent with the stabilizing effect of SD-TRK in this ill-conditioned doubly noisy setting.

\subsubsection{Controlled Known-Target Comparison} \label{ex515}
A controlled known-target experiment is conducted to control for stochastic
sampling variability and provide a direct point-to-point comparison between the
standard TRK and SD-TRK frameworks. In each Monte Carlo realization, the two
methods are run on the same ground-truth tensor \(\mathcal X^\star\), the same
observed operator \(\tilde{\mathcal A}\), the same observed right-hand side
\(\tilde{\mathcal B}\), and the same randomized row-selection sequence
\(\mathcal I_T=\{i_1,i_2,\dots,i_T\}\). Across different realizations, the system
and row-selection sequence are regenerated independently. Consequently, the
comparison is paired within each realization, so discrepancies between the two
trajectories can be attributed to the intrinsic update mechanisms rather than to
different sampled systems or row sequences.

To construct a challenging recovery scenario, the tensor system is generated by the same lateral-slice scaling construction as in Subsubsection~\ref{ex514}, with scale factors decaying geometrically from \(1\) to \(3\times10^{-2}\). As in Subsubsection~\ref{ex514}, we use a sparse structural configuration for the ground-truth tensor \( \mathcal{X}^\star\).
Specifically, the dominant signal energy is concentrated entirely within the leading
\(5\) active rows, while the remaining rows are set identically to zero. This sparse
construction is motivated by practical tensor applications in which only a limited
number of dominant structural components contain meaningful information, whereas
the remaining components primarily consist of weak or noise-dominated directions.

The observation tensor is contaminated with substantial perturbation noise
\((\delta_B=0.5)\), while the operator tensor is subjected to a relatively minor perturbation
\((\delta_A=0.01)\). For the proposed SD-TRK framework, the relaxation parameter is
set to \(\omega=0.8\) together with a damping scale
\(\lambda_{\mathrm{scale}}=50.0\). The corresponding slice-dependent damping
parameters are again computed from \eqref{eq:lambda_scale_rule} with
\(\lambda_{\mathrm{floor}}=10^{-6}\). Both algorithms are run for
\(T=2000\) iterations, and all convergence trajectories are reported as averages
over \(10\) paired Monte Carlo realizations.

Figure~\ref{fig:known_target_comparison} presents the resulting convergence trajectories on a logarithmic scale, directly comparing the true reconstruction errors of standard TRK and SD-TRK together with the observable SD-TRK residual. The standard TRK residual is omitted from the figure for readability.

\begin{figure}[ht]
    \centering
    \includegraphics[width=0.75\linewidth]{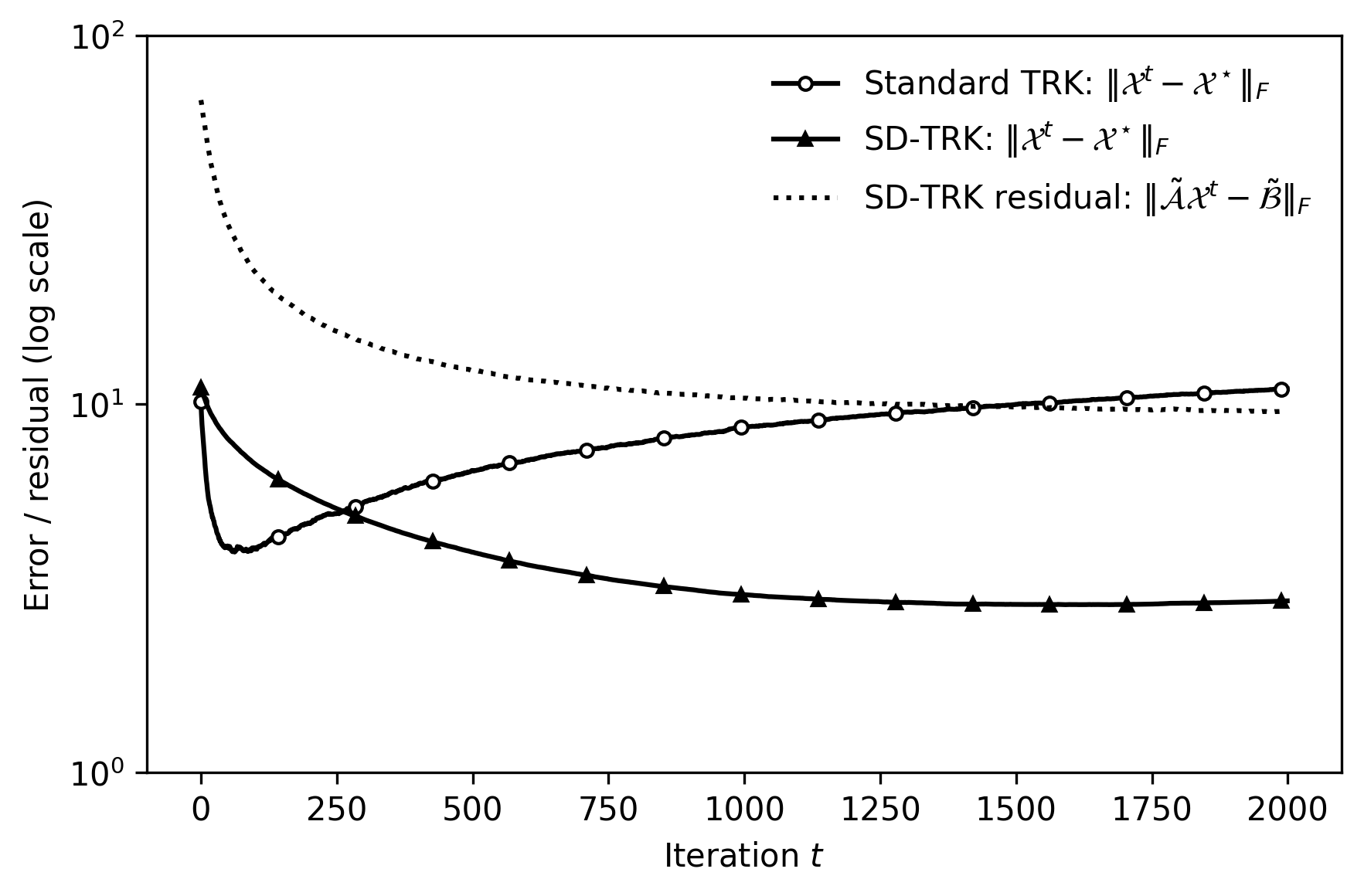}
    \caption{Controlled known-target comparison between standard TRK and SD-TRK under a sparse ill-conditioned doubly noisy setting.}
    \label{fig:known_target_comparison}
\end{figure}

The resulting trajectories show a clear difference between the two methods in a
sparse-target, ill-conditioned doubly noisy setting. As shown in
Figure~\ref{fig:known_target_comparison}, the standard TRK error initially
decreases during the early iterations. It then increases and exhibits pronounced
semi-convergence. Since the target tensor contains large zero-valued regions, the
standard TRK iteration attempts to fit noise-contaminated components in directions
where no true signal is present. Near-zero singular values further amplify the
observational perturbation, causing the true reconstruction error
\(\|\mathcal X^t-\mathcal X^\star\|_F\) to increase toward an asymptotic level of
approximately \(10.9761\).

In contrast, SD-TRK remains stable under the same row-sampling sequence. The frequency-dependent damping mechanism suppresses weak spectral directions and limits the amplification of noise-dominated components. Consequently, the true reconstruction error decreases nearly monotonically and stabilizes near an error floor of approximately \(2.9251\). Thus, the final reconstruction error is approximately \(3.75\) times smaller than that of standard TRK.

Moreover, the observable SD-TRK residual
\(\|\tilde{\mathcal A}\mathcal X^t-\tilde{\mathcal B}\|_F\) remains closely aligned with the stabilization phase of the true reconstruction error. 
This comparison is meaningful because the two methods use the same row-sampling sequence within each paired realization.
The experiment therefore suggests that the spectral damping mechanism suppresses noise amplification in weak spectral components while preserving the dominant sparse structure of the target tensor.

\subsubsection{Reference Horizon vs. Empirical Stabilization} \label{ex516}
We finally investigate the relationship between a damped reference horizon and the empirical stabilization behavior of SD-TRK under varying noise levels. We adopt the general doubly noisy mathematical framework introduced in Subsubsection~\ref{ex512}, setting symmetric noise scales such that $\delta_A = \delta_B = \delta$, with
\[
\delta \in \{0,10^{-3},5\cdot 10^{-3},10^{-2},5\cdot 10^{-2}\}.
\]
For each value of \(\delta\), we run \(10\) independent trials, each with \(1500\) iterations. The SD-TRK parameters are fixed at
\(\omega=0.9\) and \(\lambda_{\mathrm{scale}}=5.0\), with
\(\lambda_{\mathrm{floor}}=10^{-6}\) in \eqref{eq:lambda_scale_rule}. 
For each
trial, we compute the average of \(\|\mathcal X^t-\mathcal X^{LS}\|_F\) over the last \(100\) iterations,
and then average these values over the \(10\) trials. We refer to this quantity as
the last-100 least-squares error. For comparison, we also plot the damped reference horizon

\[
H_{\mathrm{damp}}
:=
\max_{1\le k\le n_3}
\frac{
\|\hat E_A^{(k)}\hat X^{LS,(k)}-\hat E_B^{(k)}\|_F
}{\sqrt{
\sigma_{\min}^2(\hat{\tilde A}^{(k)})+\lambda^{(k)}}
}.
\]

This quantity is not a stopping threshold derived from the theorem. It is a conservative
slice-wise proxy for the size of the perturbation term after damped spectral
normalization. It can be viewed as a damped and slice-wise analogue of the perturbation term in the reference scale \eqref{eq:lsbound}, stated on the error-norm scale rather than the squared-error scale.
We compare it with the last-100 least-squares error because both
quantities are measured in the same error-norm scale and both reflect the
noise-dependent stabilization level of the observed-system iterates.

\begin{figure}[ht]
    \centering
    \includegraphics[width=0.75\linewidth]{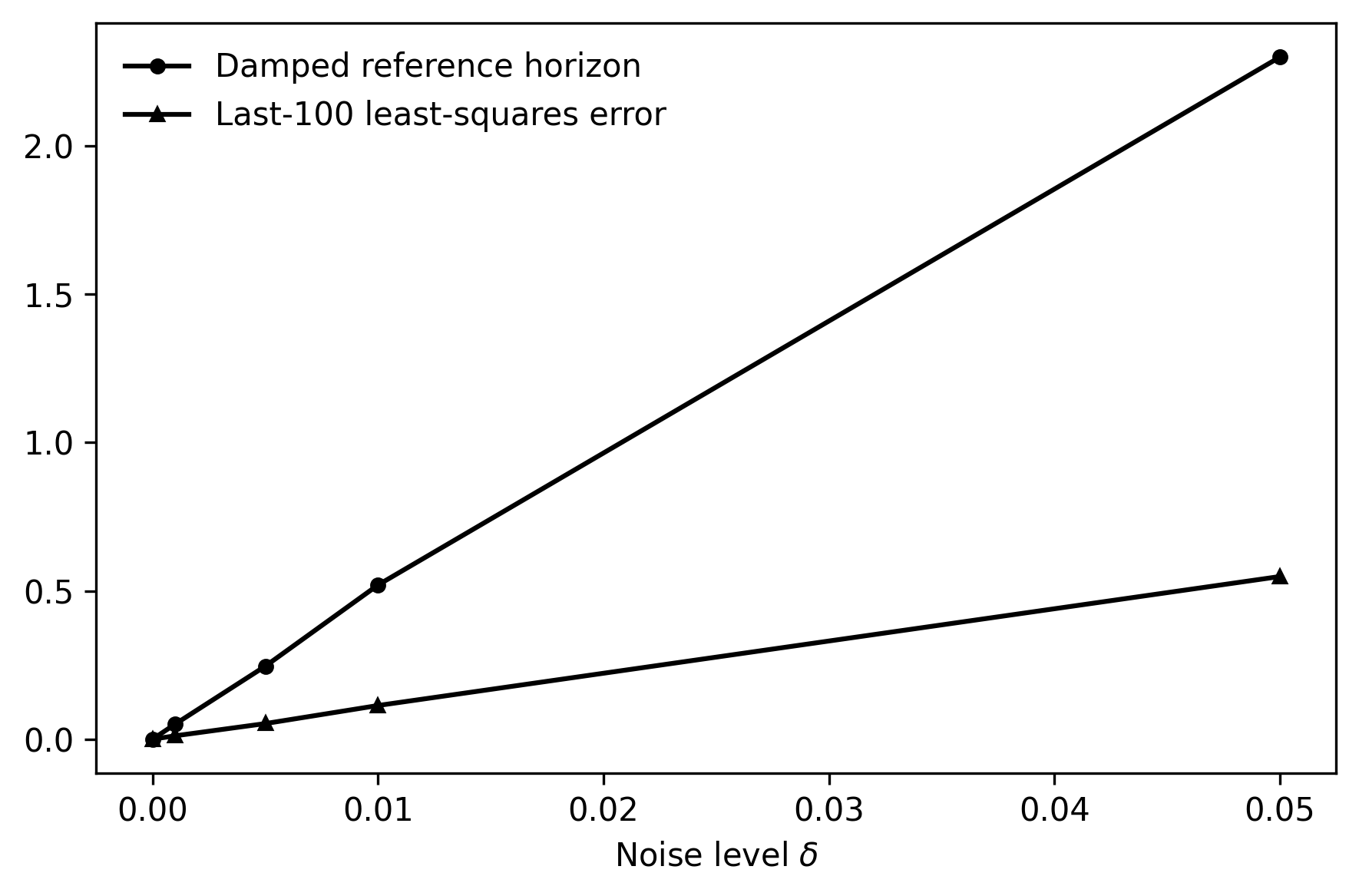}
    \caption{Damped reference horizon and last-100 least-squares error under different noise levels \(\delta\).}
\label{fig:noise_scaling}
\end{figure}

Figure~\ref{fig:noise_scaling} shows that both the damped reference horizon and the last-100 least-squares error increase as the positive noise level \(\delta\) increases.
For the positive noise levels tested, the damped reference horizon remains above the last-100 least-squares error, which is consistent with its role as a conservative reference horizon.
At \(\delta=0\), the reference horizon is zero by construction and serves as the noiseless reference case.
This gap is expected because the damped reference horizon is computed using
slice-wise spectral information, while the observed last-100 least-squares error reflects the actual behavior of SD-TRK iterates.

\subsection{Image Reconstruction Comparison}
\label{sec:image_application}

We next consider a two-pass image reconstruction problem to compare standard TRK
and SD-TRK. 
The goal is not to
claim state-of-the-art image restoration performance, but to examine the behavior
of the two Kaczmarz updates under the same reconstruction pipeline.

Unlike the synthetic experiments in Subsubsections~\ref{ex513}--\ref{ex516}, which
use slice-dependent damping parameters \(\lambda^{(k)}\), this image experiment
uses a fixed scalar damping parameter \(\lambda=10^{-2}\). It should therefore be
interpreted as a comparison of damped and undamped Kaczmarz updates in a two-pass
reconstruction setting, not as a direct validation of the slice-dependent damping
rule in \eqref{eq:lambda_scale_rule}. In fact, since the blur operators below have nonzero entries only in the first frontal slice, all Fourier frontal slices of each operator coincide. Slice-dependent damping therefore offers no additional flexibility in this setting, and a scalar damping parameter is the natural choice.

We model the degradation using a separable Gaussian blur. Let
\(A_{\mathrm{mat}}^r\) and \(A_{\mathrm{mat}}^c\) denote the two-dimensional
noise-free blur matrices acting along the row and column directions, respectively.
These matrices are embedded into t-product tensors by placing them in the first
frontal slice and setting the remaining frontal slices to zero:
\[
\mathcal A^r_{::1}=A_{\mathrm{mat}}^r,
\qquad
\mathcal A^c_{::1}=A_{\mathrm{mat}}^c,
\qquad
\mathcal A^r_{::k}=\mathcal A^c_{::k}=0
\quad \text{for } k=2,3.
\]
The observed blur operators are formed by perturbing only these underlying
two-dimensional blur matrices before embedding them as tensors. Specifically,
\[
\tilde{\mathcal A}^r_{::1}
=
A_{\mathrm{mat}}^r+\delta_A G^r_{\mathrm{mat}},
\qquad
\tilde{\mathcal A}^c_{::1}
=
A_{\mathrm{mat}}^c+\delta_A G^c_{\mathrm{mat}},
\]
and
\[
\tilde{\mathcal A}^r_{::k}
=
\tilde{\mathcal A}^c_{::k}
=
0
\quad \text{for } k=2,3.
\]
Here, the entries of \(G^r_{\mathrm{mat}}\) and \(G^c_{\mathrm{mat}}\) are independently
sampled from the standard Gaussian distribution, and \(\delta_A\) controls the
operator perturbation level.

For the image experiment only, we use \(\mathcal Z^{\top_{12}}\) to denote
the tensor obtained by transposing each frontal slice of \(\mathcal Z\), namely
\[
(\mathcal Z^{\top_{12}})_{ijk}=\mathcal Z_{jik}.
\]
This operation swaps only the row and column image directions and does not reverse
the third dimension. It should not be confused with the t-product conjugate
transpose \(\mathcal Z^*\).

Using the noise-free blur operators, the blurred image is generated as
\[
\mathcal{B}^{\mathrm{blur}}
=
\left(
\mathcal{A}^{c}
(\mathcal{A}^{r}\mathcal{X}^{\textrm{true}})^{\top_{12}}
\right)^{\top_{12}}.
\]
The observed image is then obtained by adding Gaussian noise to this blurred image:
\begin{equation}\label{eq:doublyimage}
\tilde{\mathcal B}
=
\mathcal B^{\mathrm{blur}}
+
\delta_B\mathcal G^{B}
=
\left(
{\mathcal A}^{c}         
({\mathcal A}^{r}\mathcal X^{\textrm{true}})^{\top_{12}}
\right)^{\top_{12}}
+
\delta_B\mathcal G^{B}.
\end{equation}
Here, the entries of \(\mathcal G^{B}\) are independently sampled from the
standard Gaussian distribution, and \(\delta_B\) controls the image noise level.
In this image pipeline, the degraded data are generated with the noise-free blur
operators \(\mathcal A^r\) and \(\mathcal A^c\), while the algorithms are supplied
with the perturbed observed operators \(\tilde{\mathcal A}^r\) and
\(\tilde{\mathcal A}^c\). Thus, \(\delta_A\) creates a mismatch between the
data-generating operators and the operators available to the algorithms, in direct
analogy with the doubly noisy model of Section~\ref{sec:error}, in which the data
are generated by the noiseless tensor \(\mathcal A\) but only the observed tensor
\(\tilde{\mathcal A}\) is available.

The model \eqref{eq:doublyimage} gives the blurred and noisy image available to the algorithm. 
Given $\tilde{\mathcal{A}}^{r}$, $\tilde{\mathcal{A}}^{c}$, and $\tilde{\mathcal{B}}$, our goal is to recover an approximation of the clean image tensor $\mathcal{X}^{\textrm{true}}$. 
Since the blur in \eqref{eq:doublyimage} is separable, we apply SD-TRK in two directional passes. 
First, we apply SD-TRK to the row-direction system
\begin{equation}
\label{eq:row_image_system}
\tilde{\mathcal{A}}^{r}\mathcal{X}
=
\tilde{\mathcal{B}},
\end{equation}
and denote the resulting intermediate reconstruction by $\mathcal{X}^{\mathrm{row}}$. 
Then, after transposing $\mathcal{X}^{\mathrm{row}}$, we apply SD-TRK to the column-direction system
\begin{equation}
\label{eq:col_image_system}
\tilde{\mathcal{A}}^{c}\mathcal{Y}
=
(\mathcal{X}^{\mathrm{row}})^{\top_{12}}.
\end{equation}
The final reconstruction is defined as
\[
\mathcal{X}^{\mathrm{rec}}=\mathcal{Y}^{\top_{12}}.
\]

\subsubsection{Experimental setting}

In the experiment, $\mathcal{X}^{\textrm{true}}$ is an RGB image of the Geisel Library at UC San Diego, resized and normalized so that all pixel values lie in $[0,1]$. 
The processed image has size  $256\times170\times3$, where the third dimension corresponds to the RGB color channels. 
The image degradation follows a standard deblurring framework \cite{Hansen2006}. 
The underlying row and column blur matrices \(A_{\mathrm{mat}}^r\) and
\(A_{\mathrm{mat}}^c\) are generated using Gaussian smoothing
\cite{GonzalezWoods2018}, with standard deviation \(2.0\) and truncation radius \(8\),
and are then embedded as t-product operators as described above.
The perturbation levels are set to $\delta_A=10^{-4}$ for the blur operators and $\delta_B=10^{-2}$ for the observed image. 
In the implementation, pixel values are clipped to the range $[0,1]$ after adding image noise.

We then implement both standard TRK and SD-TRK for the two directional
systems
\eqref{eq:row_image_system} and \eqref{eq:col_image_system}. The two
methods use the same two-pass reconstruction procedure, the same sampling rule, and the same number of iterations in each pass. For SD-TRK, both passes use relaxation parameter \(\omega=0.30\), damping parameter \(\lambda=10^{-2}\), and \(1200\) iterations.
The standard TRK baseline is obtained by using the corresponding
undamped update,
that is, \(\omega=1\) and \(\lambda=0\), with the same iteration count. Since the blurred observation already provides a natural approximation of the clean image, both directional passes are initialized from their degraded right-hand sides, that is, $\mathcal{X}^0=\tilde{\mathcal{B}}$ for the row pass and $\mathcal{X}^0=(\mathcal{X}^{\mathrm{row}})^{\top_{12}}$ for the column pass.

To measure how close an image is to the clean image $\mathcal{X}^{\textrm{true}}$, we use two metrics. 
For an image tensor $\mathcal{Z}$, the first metric is the relative error
\begin{equation}
\label{eq:relerror}
\frac{\|\mathcal{Z}-\mathcal{X}^{\textrm{true}}\|_F}
{\|\mathcal{X}^{\textrm{true}}\|_F}.
\end{equation}
The second metric is the peak signal-to-noise ratio (PSNR), a standard image quality metric in image processing \cite{GonzalezWoods2018,imagequality2004}.

\subsubsection{Numerical results}

We now compare the performance of the standard TRK and the proposed SD-TRK using quantitative reconstruction metrics and residual evolution. 
For SD-TRK, we use relaxation parameter \(\omega=0.30\) and a fixed scalar damping
parameter \(\lambda = 10^{-2}\) throughout the two-pass reconstruction pipeline.

While standard TRK may be sensitive to noise amplification in noisy and inconsistent settings, SD-TRK incorporates a damped normalization term that regularizes the inverse step and improves stability across the two directional passes. To quantitatively assess the reconstruction quality, we compute two standard image fidelity metrics: the relative reconstruction error and the Peak Signal-to-Noise Ratio (PSNR). Table~\ref{tab:image_metrics} summarizes the final quantitative image results recorded at the end of \(1200\) iterations per directional pass for the degraded observation and the two reconstructed outputs.

\begin{table}[ht]
\centering
\caption{Quantitative image results for the degraded observation and reconstructed outputs}
\label{tab:image_metrics}
\renewcommand{\arraystretch}{1.25}
\begin{tabular}{c|c|c}
\hline
Image tensor or method & Relative error \eqref{eq:relerror} & PSNR (dB) \\
\hline
Degraded observation \(\tilde{\mathcal B}\) & 0.1289 & 22.5856 \\
Standard TRK reconstruction & 0.1246 & 22.8824 \\
SD-TRK reconstruction \(\mathcal X^{\mathrm{rec}}\) & 0.1025 & 24.5723 \\
\hline
\end{tabular}
\end{table}

The quantitative metrics indicate that, under the same two-pass reconstruction procedure, SD-TRK achieves better reconstruction quality than the standard TRK baseline in this experiment. Specifically, SD-TRK achieves a lower relative error of \(0.1025\) and a higher PSNR of \(24.5723\)~dB, showing a quantitative improvement over both the standard TRK and the initial degraded observation \(\tilde{\mathcal B}\). 

\begin{figure}[ht]
    \centering
    \includegraphics[width=\linewidth]{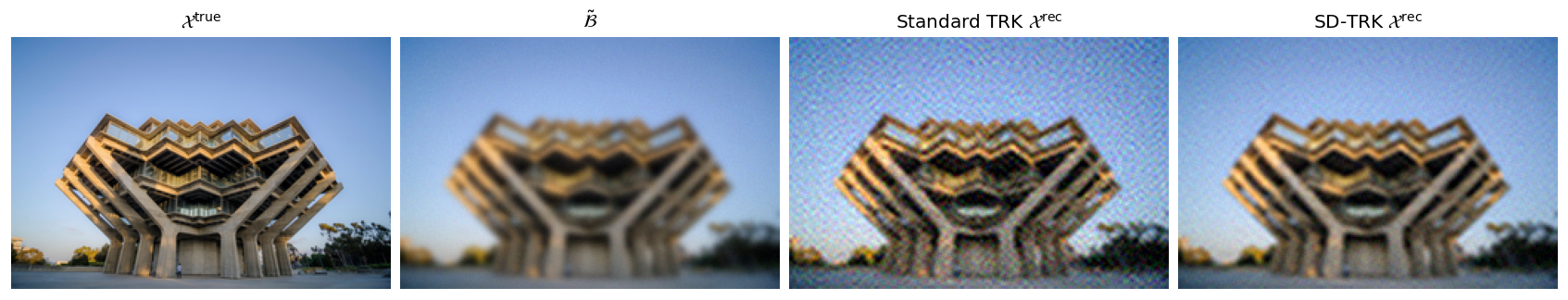}
    \caption{Image restoration results}
    \label{fig:image_reconstruction}
\end{figure}

Figure~\ref{fig:image_reconstruction} displays the clean image, the degraded observation, the standard TRK, and the SD-TRK reconstruction. Although the visual differences between the standard and damped updates appear modest due to the well-conditioned nature of the underlying blur operators, the proposed damping mechanism reduces noise propagation across successive directional passes. 

To examine the iterative behavior of the two methods, we track the residual norms for the row-directional and column-directional systems. For visualization, the Frobenius norm of each residual is recorded at regular intervals during the iteration process.

\begin{figure}[ht]
    \centering
    \includegraphics[width=\linewidth]{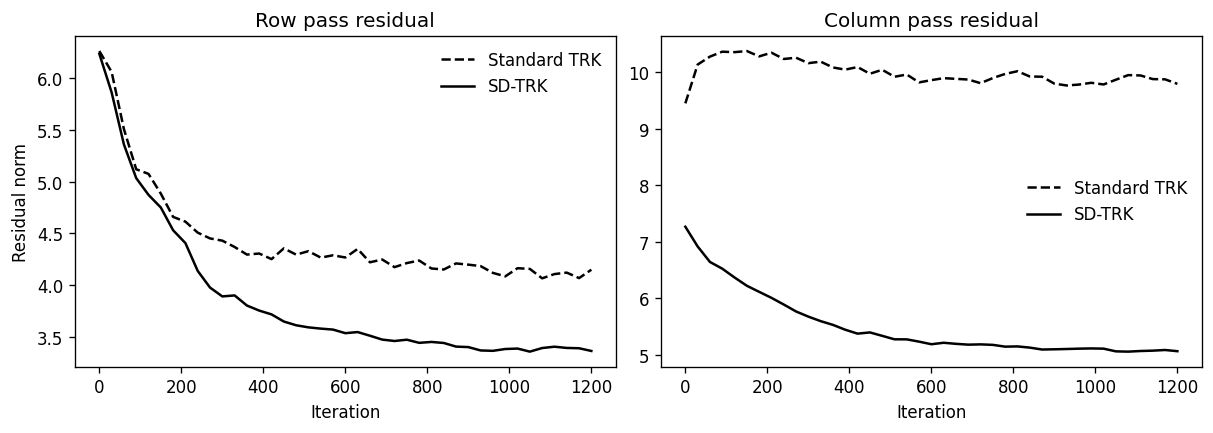}
    \caption{Residual decay for the two directional systems}
    \label{fig:residual_decay}
\end{figure}

Figure~\ref{fig:residual_decay} shows the residual evolution for the row-directional and column-directional systems under both standard TRK and SD-TRK. 
The residual trajectories stabilize near nonzero plateaus. In the row pass, both methods decrease during the early iterations, whereas in the column pass the standard TRK residual initially increases slightly before stabilizing.
Since the observed right-hand sides contain image noise and pixel values are clipped after noise is added, driving the residuals to zero is not necessarily aligned with recovering the clean image. 
In addition, the fixed finite iteration budget and, for SD-TRK, the damped updates lead to nonzero terminal residuals in the reported runs.

In this experiment, SD-TRK attains lower final residuals than the standard TRK baseline in both directional passes. This behavior is consistent with the quantitative image metrics in Table~\ref{tab:image_metrics}, where SD-TRK also achieves the smallest relative reconstruction error and the largest PSNR. These results indicate that the damped update improves the stability of the two-pass reconstruction process under this noisy image reconstruction setting.

\section{Conclusion}\label{sec:con}

In this paper, we studied tensor randomized Kaczmarz methods for tensor linear
systems under a doubly noisy perturbation model. We first analyzed standard TRK
when both the system tensor and the right-hand side tensor are perturbed. The
resulting recursion contains a contraction term and a persistent perturbation term.
This explains why standard TRK may exhibit semi-convergent or noise-limited
behavior in doubly noisy tensor systems.

We then introduced the spectrally damped tensor randomized Kaczmarz method. The
SD-TRK update combines a relaxation parameter \(\omega\) with a Tikhonov-type
damping parameter \(\lambda\). It reduces to standard TRK when \(\omega=1\) and
\(\lambda=0\). We derived an expected error recursion for SD-TRK that separates
the propagation of the previous error from the perturbation injected by the noisy
model. The bound shows a speed--robustness trade-off. Stronger damping and
underrelaxation can reduce noise amplification. They may also weaken the guaranteed
contraction. We also presented an FFT-based implementation that applies the damped
update slice-wise in the Fourier domain and allows frequency-dependent damping
parameters \(\lambda^{(k)}\).

The numerical experiments are consistent with this interpretation. In noiseless settings,
standard TRK exhibits the expected convergent behavior. In doubly noisy settings,
its error trajectories are affected by a persistent noise-dependent component. In
the well-conditioned, ill-conditioned, and controlled known-target comparisons,
SD-TRK is more stable than standard TRK in the tested settings. 
The noise-scaling experiment shows that, for the positive noise levels tested, the empirical last-100 least-squares error increases with the noise level and remains below the damped reference horizon, while the noiseless case serves as a zero-horizon reference case.
The image experiment also suggests that SD-TRK is more stable than standard TRK under the same two-pass reconstruction procedure. In the reported setting, SD-TRK improves both the relative error and PSNR compared with the degraded observation and the standard TRK reconstruction. Thus, the image experiment should be interpreted as a controlled method comparison rather than as a
state-of-the-art image restoration benchmark.

Several limitations remain. The main theoretical recursion is derived for a scalar
damping parameter \(\lambda\), while the implementation allows slice-dependent
parameters \(\lambda^{(k)}\). The experiments with slice-dependent damping should
therefore be viewed as empirical validation of the Fourier-domain implementation,
not as direct numerical instances of the scalar-parameter theorem. The present work
also uses fixed user-chosen relaxation and damping parameters. A systematic
ablation of relaxation and damping, together with adaptive parameter selection,
remains an important direction for future work.

A further direction is to move beyond damped least-squares-type stabilization and
develop robust tensor Kaczmarz methods for more structured corruption models.
Tensor total least-squares formulations are natural when perturbations in the
system tensor and the right-hand side tensor should be treated jointly. Combining
spectral damping with robust screening rules may lead to tensor Kaczmarz methods
that are stable under dense noise and robust to sparse outliers. Interquartile-range
screening of corrupted row slices or frequency components is one possible approach.
Such extensions would also clarify the relationship between spectral regularization,
semi-convergence, and robust row-action methods for doubly noisy tensor inverse
problems.

\bibliographystyle{siam}
\bibliography{references} 

\appendix

\section{Efficient t-product Computation via FFT}
\label{sec:appendix_fft}

In this appendix, we provide the mathematical background and algorithmic details for the efficient computation of the t-product. The t-product operation is fundamentally based on the convolution theorem, which allows us to transform computationally expensive convolution operations in the spatial domain into efficient element-wise multiplications in the frequency domain.

\subsection{Mathematical Foundations}
Let $\mathcal{A} \in \mathbb{C}^{n_1 \times n_2 \times n_3}$ and $\mathcal{B} \in \mathbb{C}^{n_2 \times n_4 \times n_3}$ be two tensors. 
The t-product $\mathcal{C} = \mathcal{A} \mathcal{B}$ can be interpreted as a circular convolution along the third dimension (tube fibers). 
Let $\mathbf{a}_{ij} = \mathcal{A}_{ij:}$ and $\mathbf{b}_{jl} = \mathcal{B}_{jl:}$ denote the tube fibers of size $n_3$. 
The $(i, l)$-th tube of the resulting tensor $\mathcal{C}$ is given by:
\begin{equation*}
    \mathbf{c}_{il} = \sum_{j=1}^{n_2} \mathbf{a}_{ij} \circledast \mathbf{b}_{jl},
\end{equation*}
where $\circledast$ denotes the circular convolution. 
According to the convolution theorem, this corresponds to point-wise multiplication in the Fourier domain:
\begin{equation} \label{eq:conv_theorem}
    \mathcal{F}(\mathbf{a} \circledast \mathbf{b}) = \mathcal{F}(\mathbf{a}) \odot \mathcal{F}(\mathbf{b}),
\end{equation}
where $\mathcal{F}$ denotes the Discrete Fourier Transform (DFT).

The computational efficiency arises from the property that the DFT matrix
diagonalizes any circulant matrix.
Let
\(C=\operatorname{circ}(v)\in\mathbb C^{q\times q}\) be the circulant matrix
whose first column is a vector \(v\) of length \(q\). 
Let \(F_q\) denote the unnormalized
DFT matrix under the FFT convention used here. Then
\[
C
=
F_q^{-1}\operatorname{diag}(\fft(v))F_q.
\]
Extending this property to the block-circulant matrix used in the t-product,
we have
\[
(F_{n_3}\otimes I_{n_1})
\bcirc(\mathcal A)
(F_{n_3}^{-1}\otimes I_{n_2})
=
\blkdiag(\hat{\mathcal A}_{::1},\ldots,\hat{\mathcal A}_{::n_3}).
\]

\subsection{Complexity Analysis: Derivation of Fast Fourier Transform}
To justify the efficiency improvement, we recall the Cooley--Tukey recursion. 
The standard DFT of a vector $x$ of length $q$ is defined as $\hat{x}_k = \sum_{j=0}^{q-1} x_j e^{-2\pi i \frac{jk}{q}}$. 
Assuming $q$ is a power of 2, we can split the summation into even and odd indices:
\begin{align*}
    \hat{x}_k &= \sum_{r=0}^{q/2-1} x_{2r} e^{-2\pi i \frac{rk}{q/2}} + e^{-2\pi i \frac{k}{q}} \sum_{r=0}^{q/2-1} x_{2r+1} e^{-2\pi i \frac{rk}{q/2}} \\
              &= E_k + W_q^k O_k,
\end{align*}
where $E_k$ and $O_k$ are the DFTs of the even and odd indexed parts, and $W_q^k$ is the twiddle factor. 
This recursive structure reduces the complexity to $T(q) = 2T(q/2) + \mathcal{O}(q)$, implying $T(q) = \mathcal{O}(q \log q)$.

\subsection{Implementation Algorithm}
Based on the derivation above, the t-product is implemented as in Algorithm \ref{alg:t_product_fft}. 
This reduces the complexity from
\begin{equation*}
    \mathcal{O}(n_1 n_2 n_4 n_3^2)
    \quad\textrm{to}\quad
    \mathcal{O}\bigl(n_1 n_2 n_4 n_3 + (n_1 n_2 + n_2 n_4 + n_1 n_4)\, n_3 \log n_3\bigr),
\end{equation*}
where the first term accounts for the slice-wise matrix products and the second term for the FFTs and the inverse FFT (IFFT).

We denote by $\hat{\mathcal A}=\fft(\mathcal A,[\,],3)$ and $\hat{\mathcal B}=\fft(\mathcal B,[\,],3)$ the Fourier transforms of $\mathcal A$ and $\mathcal B$ along the third dimension.
In this algorithmic description, we use the tensor-index notation
\(\hat{\mathcal A}_{::k}\) for the \(k\)-th frontal slice; equivalently,
\(\hat{\mathcal A}_{::k}=\hat A^{(k)}\), \(\hat{\mathcal B}_{::k}=\hat B^{(k)}\),
and \(\hat{\mathcal C}_{::k}=\hat C^{(k)}\).

\begin{algorithm}[ht]
\caption{Fast t-product via FFT}
\label{alg:t_product_fft}
\begin{algorithmic}[1]
\Require Tensors $\mathcal{A} \in \mathbb{C}^{n_1 \times n_2 \times n_3}$, $\mathcal{B} \in \mathbb{C}^{n_2 \times n_4 \times n_3}$
\Ensure Tensor $\mathcal{C} = \mathcal{A}  \mathcal{B} \in \mathbb{C}^{n_1 \times n_4 \times n_3}$
\State Compute the Fourier transforms
\[
\hat{\mathcal{A}} = \fft(\mathcal{A}, [\,], 3), \qquad
\hat{\mathcal{B}} = \fft(\mathcal{B}, [\,], 3)
\]
\For{$k = 1$ to $n_3$}
    \State Compute
    $
    \hat{\mathcal{C}}_{::k} = \hat{\mathcal{A}}_{::k}  \hat{\mathcal{B}}_{::k}
    $
\EndFor
\State Compute the inverse Fourier transform
\[
\mathcal{C} = \ifft(\hat{\mathcal{C}}, [\,], 3)
\]
\State \Return $\mathcal{C}$
\end{algorithmic}
\end{algorithm}

\section{Proof of Lemma~\ref{lem:orthogonality}}\label{sec:orthogonality}

Using the TRK update rule \eqref{eq:TRK} for the observed system \eqref{eq:doublytls}, we first decompose the residual at iteration $t$. By substituting $\mathcal{X}^t = \mathcal{E}^t + \mathcal{X}^\star$, we have the following.
\begin{align*}
\tilde{\mathcal{A}}_{i_t::}\mathcal{X}^t - \tilde{\mathcal{B}}_{i_t::} &= \tilde{\mathcal{A}}_{i_t::}(\mathcal{E}^t + \mathcal{X}^\star) - \tilde{\mathcal{B}}_{i_t::} \\
&= \tilde{\mathcal{A}}_{i_t::}\mathcal{E}^t + (\mathcal{A}_{i_t::} + (\mathcal E_A)_{i_t::})\mathcal{X}^\star - (\mathcal{B}_{i_t::} + (\mathcal E_B)_{i_t::}) \\
&= \tilde{\mathcal{A}}_{i_t::}\mathcal{E}^t + (\mathcal{A}_{i_t::}\mathcal{X}^\star - \mathcal{B}_{i_t::}) + ((\mathcal E_A)_{i_t::} \mathcal{X}^\star - (\mathcal E_B)_{i_t::}).
\end{align*}
Since $\mathcal{X}^\star$ is the exact solution to the noiseless system, $\mathcal{A}_{i_t::}\mathcal{X}^\star - \mathcal{B}_{i_t::} = 0$. Let $\mathcal{N}_{i_t} = (\mathcal E_A)_{i_t::} \mathcal{X}^\star - (\mathcal E_B)_{i_t::}$ denote the local noise component tensor. 
Substituting this back into the error update yields:
\begin{equation} \label{eq:error_update_tensor}
\mathcal{E}^{t+1} = \mathcal{E}^t - \tilde{\mathcal{A}}_{i_t::}^* (\tilde{\mathcal{A}}_{i_t::}\tilde{\mathcal{A}}_{i_t::}^*)^{-1} (\tilde{\mathcal{A}}_{i_t::}\mathcal{E}^t + \mathcal{N}_{i_t}).
\end{equation}

To evaluate the squared Frobenius norm \(\|\mathcal E^{t+1}\|_F^2\),
we transform the system into the Fourier domain using the notation introduced above.
Let \(\hat E^{t,(k)}\) and \(\hat N_{i_t}^{(k)}\) denote the \(k\)-th frontal slices
of \(\hat{\mathcal E}^t\) and \(\hat{\mathcal N}_{i_t}\), respectively.

By the properties of the t-product in the Fourier domain (see Fact 1 in \cite{MaMolitor2022}), the tensor operations map to slice-wise matrix operations for each frequency $k \in [n_3]$. 
Thus, equation \eqref{eq:error_update_tensor} decouples into $n_3$ independent matrix equations as follows.
\begin{equation*} 
\hat{E}^{t+1, (k)} = \hat{E}^{t, (k)} - (\hat{\tilde{A}}_{i_t:}^{(k)})^* \left( \hat{\tilde{A}}_{i_t:}^{(k)} (\hat{\tilde{A}}_{i_t:}^{(k)})^* \right)^{-1} \left( \hat{\tilde{A}}_{i_t:}^{(k)} \hat{E}^{t, (k)} + \hat{N}_{i_t}^{(k)} \right).
\end{equation*}

Let 
\begin{equation*}
    P_{i_t}^{(k)} = (\hat{\tilde{A}}_{i_t:}^{(k)})^* \left( \hat{\tilde{A}}_{i_t:}^{(k)} (\hat{\tilde{A}}_{i_t:}^{(k)})^* \right)^{-1} \hat{\tilde{A}}_{i_t:}^{(k)}.
\end{equation*} 
In the standard matrix domain, $P_{i_t}^{(k)}$ is an orthogonal projection matrix onto the row space of $\hat{\tilde{A}}_{i_t:}^{(k)}$. As an orthogonal projector, it satisfies idempotency ($(P_{i_t}^{(k)})^2 = P_{i_t}^{(k)}$) and self-adjointness ($(P_{i_t}^{(k)})^* = P_{i_t}^{(k)}$). 
We rewrite the error update for the $k$-th frontal slice as follows.
\begin{equation*}
\hat{E}^{t+1, (k)} = (I - P_{i_t}^{(k)})\hat{E}^{t, (k)} - (\hat{\tilde{A}}_{i_t:}^{(k)})^* \left( \hat{\tilde{A}}_{i_t:}^{(k)} (\hat{\tilde{A}}_{i_t:}^{(k)})^* \right)^{-1} \hat{N}_{i_t}^{(k)}.
\end{equation*}

To compute the squared Frobenius norm of $\hat{E}^{t+1, (k)}$, we expand the norm of the right-hand side. 
The cross-term is 
\begin{equation}\label{eq:cross}
    \langle (I - P_{i_t}^{(k)})\hat{E}^{t, (k)}, (\hat{\tilde{A}}_{i_t:}^{(k)})^* \left( \hat{\tilde{A}}_{i_t:}^{(k)} (\hat{\tilde{A}}_{i_t:}^{(k)})^* \right)^{-1} \hat{N}_{i_t}^{(k)} \rangle.
\end{equation}
Since $P_{i_t}^{(k)}$ is self-adjoint, we have $(I - P_{i_t}^{(k)})^* = I - P_{i_t}^{(k)}$.
To compute the cross-term \eqref{eq:cross}, we use the adjoint identity for the inner product.
Then \eqref{eq:cross} becomes
\begin{equation*}
    \langle \hat{E}^{t, (k)}, (I - P_{i_t}^{(k)})^*(\hat{\tilde{A}}_{i_t:}^{(k)})^* \left( \hat{\tilde{A}}_{i_t:}^{(k)} (\hat{\tilde{A}}_{i_t:}^{(k)})^* \right)^{-1} \hat{N}_{i_t}^{(k)} \rangle.
\end{equation*}
Since
\begin{align*}
(I - P_{i_t}^{(k)})^* (\hat{\tilde{A}}_{i_t:}^{(k)})^* &= (I - P_{i_t}^{(k)}) (\hat{\tilde{A}}_{i_t:}^{(k)})^* \\
&= (\hat{\tilde{A}}_{i_t:}^{(k)})^* - P_{i_t}^{(k)} (\hat{\tilde{A}}_{i_t:}^{(k)})^* \\
&= (\hat{\tilde{A}}_{i_t:}^{(k)})^* - (\hat{\tilde{A}}_{i_t:}^{(k)})^* \left( \hat{\tilde{A}}_{i_t:}^{(k)} (\hat{\tilde{A}}_{i_t:}^{(k)})^* \right)^{-1} \left[ \hat{\tilde{A}}_{i_t:}^{(k)} (\hat{\tilde{A}}_{i_t:}^{(k)})^* \right] \\
&= (\hat{\tilde{A}}_{i_t:}^{(k)})^* - (\hat{\tilde{A}}_{i_t:}^{(k)})^* I \\
&= 0,
\end{align*}
the projected error $(I - P_{i_t}^{(k)})\hat{E}^{t, (k)}$ and the noise term $(\hat{\tilde{A}}_{i_t:}^{(k)})^* \left( \hat{\tilde{A}}_{i_t:}^{(k)} (\hat{\tilde{A}}_{i_t:}^{(k)})^* \right)^{-1} \hat{N}_{i_t}^{(k)}$ are orthogonal. 
Consequently, the Pythagorean theorem holds in the Fourier domain as follows.
\begin{equation} \label{eq:slice_pythagorean}
\|\hat{E}^{t+1, (k)}\|_F^2 = \|(I - P_{i_t}^{(k)})\hat{E}^{t, (k)}\|_F^2 + \left\lVert (\hat{\tilde{A}}_{i_t:}^{(k)})^* \left( \hat{\tilde{A}}_{i_t:}^{(k)} (\hat{\tilde{A}}_{i_t:}^{(k)})^* \right)^{-1} \hat{N}_{i_t}^{(k)} \right\rVert_F^2.
\end{equation}

Finally, we sum equation \eqref{eq:slice_pythagorean} over all frequencies $k \in [n_3]$ and divide by $n_3$. 
By Parseval's theorem, the scaled sum of the squared Frobenius norms of the Fourier slices is equal to the squared Frobenius norm of the original spatial tensor, i.e., $\|\mathcal{Y}\|_F^2 = \frac{1}{n_3} \sum_{k=1}^{n_3} \|\hat{Y}^{(k)}\|_F^2$ \cite{MaMolitor2022}.
Note that the set of slice-wise projection matrices $\{P_{i_t}^{(k)}\}_{k=1}^{n_3}$ precisely corresponds to the tensor-domain projection operator $\tilde{\mathcal{P}}_{i_t} = \tilde{\mathcal{A}}_{i_t::}^* (\tilde{\mathcal{A}}_{i_t::}\tilde{\mathcal{A}}_{i_t::}^*)^{-1} \tilde{\mathcal{A}}_{i_t::}$. 
Applying this relation to each term maps the decomposed Fourier equations back to the spatial tensor domain, yielding the desired orthogonal decomposition:
\begin{equation*}
\|\mathcal{E}^{t+1}\|_F^2 = \|(\mathcal{I} - \tilde{\mathcal{P}}_{i_t})\mathcal{E}^t\|_F^2 + \left\lVert \tilde{\mathcal{A}}_{i_t::}^* (\tilde{\mathcal{A}}_{i_t::}\tilde{\mathcal{A}}_{i_t::}^*)^{-1} \mathcal{N}_{i_t} \right\rVert_F^2.
\end{equation*}

\section{Proof of Theorem~\ref{thm:upper_recursion}}\label{sec:upper_recursion}

Taking the total expectation of the orthogonal decomposition established in Lemma~\ref{lem:orthogonality} with respect to the random variable $i_t \in [n_1]$, which is sampled independently at each iteration according to the probability $p_i = \|\tilde{\mathcal{A}}_{i::}\|_F^2 / \|\tilde{\mathcal{A}}\|_F^2$ for $i \in [n_1]$, we have the following.
\begin{equation} \label{eq:expectation_decomposition}
\mathbb{E}\left[ \|\mathcal{E}^{t+1}\|_F^2 \right]
=
\mathbb{E}\left[ \|(\mathcal{I} - \tilde{\mathcal{P}}_{i_t})\mathcal{E}^t\|_F^2 \right]
+
\mathbb{E}\left[ \left\| \tilde{\mathcal{A}}_{i_t::}^* (\tilde{\mathcal{A}}_{i_t::} \tilde{\mathcal{A}}_{i_t::}^*)^{-1} \mathcal{N}_{i_t} \right\|_F^2 \right],
\end{equation}
where $\mathcal{N}_{i_t} = (\mathcal E_A)_{i_t::} \mathcal{X}^\star - (\mathcal E_B)_{i_t::}$.

For the first term, we apply the Pythagorean identity
\[
\|(\mathcal{I} - \tilde{\mathcal{P}}_{i_t})\mathcal{E}^t\|_F^2
=
\|\mathcal{E}^t\|_F^2 - \|\tilde{\mathcal{P}}_{i_t} \mathcal{E}^t\|_F^2.
\]
To calculate the expected reduction $\mathbb{E}[\|\tilde{\mathcal{P}}_{i_t} \mathcal{E}^t\|_F^2]$, we utilize the selection probability $p_i = \|\tilde{\mathcal{A}}_{i::}\|_F^2 / \|\tilde{\mathcal{A}}\|_F^2$. 
Specifically, since terms with \(p_i=0\) do not contribute to the expectation, it is enough to consider \(i\in\mathcal I_+\). For each such \(i\), its squared Frobenius norm is evaluated via Parseval's theorem as follows.
\begin{align*}
\|\tilde{\mathcal{P}}_i \mathcal{E}^t\|_F^2
&=
\frac{1}{n_3} \sum_{k=1}^{n_3}
\left\|
\hat{\tilde{A}}_{i:}^{(k)*}
\left( \hat{\tilde{A}}_{i:}^{(k)} \hat{\tilde{A}}_{i:}^{(k)*} \right)^{-1}
\hat{\tilde{A}}_{i:}^{(k)} \hat E^{t,(k)}
\right\|_F^2 \\
&=
\frac{1}{n_3} \sum_{k=1}^{n_3}
\left\|
\hat{\tilde{A}}_{i:}^{(k)*}
\frac{\hat{\tilde{A}}_{i:}^{(k)} \hat E^{t,(k)}}{\|\hat{\tilde{A}}_{i:}^{(k)}\|_F^2}
\right\|_F^2 \\
&=
\frac{1}{n_3} \sum_{k=1}^{n_3}
\frac{\|\hat{\tilde{A}}_{i:}^{(k)*}\|_F^2 \, \|\hat{\tilde{A}}_{i:}^{(k)} \hat E^{t,(k)}\|_F^2}{\|\hat{\tilde{A}}_{i:}^{(k)}\|_F^4} \\
&=
\frac{1}{n_3} \sum_{k=1}^{n_3}
\frac{\|\hat{\tilde{A}}_{i:}^{(k)} \hat E^{t,(k)}\|_F^2}{\|\hat{\tilde{A}}_{i:}^{(k)}\|_F^2}.
\end{align*}

For the contraction term, we use the following consequence of Parseval's theorem:
for every \(k\in[n_3]\),
\[
\|\hat{\tilde A}_{i:}^{(k)}\|_F^2
\le
\sum_{\ell=1}^{n_3}
\|\hat{\tilde A}_{i:}^{(\ell)}\|_F^2
=
n_3\|\tilde{\mathcal A}_{i::}\|_F^2.
\]

Therefore,
\begin{equation}\label{eq:contraction_freq_bound}
\frac{1}{\|\hat{\tilde A}_{i:}^{(k)}\|_F^2}
\ge
\frac{1}{n_3\|\tilde{\mathcal A}_{i::}\|_F^2}
\qquad \text{for all } k\in[n_3].
\end{equation}

Substituting \eqref{eq:contraction_freq_bound} into the previous expression yields
\begin{align*}
\|\tilde{\mathcal{P}}_i \mathcal{E}^t\|_F^2
&=
\frac{1}{n_3} \sum_{k=1}^{n_3}
\frac{\|\hat{\tilde{A}}_{i:}^{(k)} \hat E^{t,(k)}\|_F^2}
{\|\hat{\tilde{A}}_{i:}^{(k)}\|_F^2} \\
&\ge
\frac{1}{n_3\|\tilde{\mathcal{A}}_{i::}\|_F^2}
\left(
\frac{1}{n_3} \sum_{k=1}^{n_3}
\|\hat{\tilde{A}}_{i:}^{(k)} \hat E^{t,(k)}\|_F^2
\right).
\end{align*}
By Parseval's theorem,
\[
\frac{1}{n_3} \sum_{k=1}^{n_3}
\|\hat{\tilde{A}}_{i:}^{(k)} \hat E^{t,(k)}\|_F^2
=
\|\tilde{\mathcal{A}}_{i::} \mathcal{E}^t\|_F^2.
\]
Therefore,
\[
\|\tilde{\mathcal{P}}_i \mathcal{E}^t\|_F^2
\ge
\frac{\|\tilde{\mathcal{A}}_{i::} \mathcal{E}^t\|_F^2}
{n_3\|\tilde{\mathcal{A}}_{i::}\|_F^2}.
\]

Taking total expectation and using the sampling probabilities, we obtain
\begin{align*}
\mathbb{E}\left[ \|\tilde{\mathcal{P}}_{i_t} \mathcal{E}^t\|_F^2 \right]
&=
\mathbb{E}\left[ \sum_{i\in\mathcal I_+} p_i \|\tilde{\mathcal{P}}_i \mathcal{E}^t\|_F^2 \right] \\
&\ge
\mathbb{E}\left[
\sum_{i\in\mathcal I_+}
\frac{\|\tilde{\mathcal{A}}_{i::}\|_F^2}{\|\tilde{\mathcal{A}}\|_F^2}
\left(
\frac{\|\tilde{\mathcal{A}}_{i::} \mathcal{E}^t\|_F^2}
{n_3\|\tilde{\mathcal{A}}_{i::}\|_F^2}
\right)
\right] \\
&=
\mathbb{E}\left[
\frac{1}{n_3\|\tilde{\mathcal{A}}\|_F^2}
\sum_{i=1}^{n_1}
\|\tilde{\mathcal{A}}_{i::} \mathcal{E}^t\|_F^2
\right].
\end{align*}

Next, the row-slice norms add up to the norm of the full product:
\[
\sum_{i=1}^{n_1} \|\tilde{\mathcal{A}}_{i::} \mathcal{E}^t\|_F^2
=
\|\tilde{\mathcal{A}} \mathcal{E}^t\|_F^2.
\]

Moreover, by the definition of the t-product,
\[
\unfold(\tilde{\mathcal{A}}\mathcal{E}^t)
=
\bcirc(\tilde{\mathcal{A}})\unfold(\mathcal{E}^t).
\]
Since the unfold operator only rearranges the entries, it preserves the Frobenius norm. Hence,
\[
\|\tilde{\mathcal{A}}\mathcal{E}^t\|_F^2
=
\|\unfold(\tilde{\mathcal{A}}\mathcal{E}^t)\|_F^2
=
\|\bcirc(\tilde{\mathcal{A}})\unfold(\mathcal{E}^t)\|_F^2.
\]
Furthermore,
\[
\|\bcirc(\tilde{\mathcal{A}})\unfold(\mathcal{E}^t)\|_F^2
\ge
\sigma_{\min}^2(\bcirc(\tilde{\mathcal{A}}))
\|\unfold(\mathcal{E}^t)\|_F^2
=
\sigma_{\min}^2(\bcirc(\tilde{\mathcal{A}}))
\|\mathcal{E}^t\|_F^2.
\]
Therefore,
\[
\mathbb{E}\left[ \|\tilde{\mathcal{P}}_{i_t} \mathcal{E}^t\|_F^2 \right]
\ge
\frac{\sigma_{\min}^2(\bcirc(\tilde{\mathcal{A}}))}
{n_3\|\tilde{\mathcal{A}}\|_F^2}
\mathbb{E}\left[\|\mathcal{E}^t\|_F^2\right].
\]
This yields the expected contraction bound as follows.
\[
\mathbb{E}\left[
\|(\mathcal I-\tilde{\mathcal P}_{i_t})\mathcal E^t\|_F^2
\right]
\le
\left(
1-
\frac{\sigma_{\min}^2(\bcirc(\tilde{\mathcal{A}}))}
{n_3\|\tilde{\mathcal{A}}\|_F^2}
\right)
\mathbb{E}\left[\|\mathcal E^t\|_F^2\right].
\]

For the second term in \eqref{eq:expectation_decomposition}, we first apply Parseval's theorem to evaluate the squared Frobenius norm of the projected noise in the Fourier domain as follows.
\begin{align*}
\left\| \tilde{\mathcal{A}}_{i::}^* (\tilde{\mathcal{A}}_{i::} \tilde{\mathcal{A}}_{i::}^*)^{-1} \mathcal{N}_i \right\|_F^2 
&= \frac{1}{n_3} \sum_{k=1}^{n_3} \left\| \hat{\tilde{A}}_{i:}^{(k)*} \left( \hat{\tilde{A}}_{i:}^{(k)} \hat{\tilde{A}}_{i:}^{(k)*} \right)^{-1} \hat{N}_i^{(k)} \right\|_F^2 \\
&= \frac{1}{n_3} \sum_{k=1}^{n_3} \left\| \hat{\tilde{A}}_{i:}^{(k)*} \frac{\hat{N}_i^{(k)}}{\|\hat{\tilde{A}}_{i:}^{(k)}\|_F^2} \right\|_F^2 \\
&= \frac{1}{n_3} \sum_{k=1}^{n_3} \frac{\|\hat{N}_i^{(k)}\|_F^2}{\|\hat{\tilde{A}}_{i:}^{(k)}\|_F^2},
\end{align*}

where the last equality follows from the outer-product identity
\[
\|\hat{\tilde A}_{i:}^{(k)*}\hat N_i^{(k)}\|_F^2
=
\|\hat{\tilde A}_{i:}^{(k)}\|_F^2
\|\hat N_i^{(k)}\|_F^2.
\]

For the perturbation term, we need an upper bound on the reciprocal Fourier-slice row norms.
By the definition of \(\gamma_i\) and Parseval's theorem, we obtain

\[
\frac{1}{\|\hat{\tilde{A}}_{i:}^{(k)}\|_F^2}
\le
\frac{1}{\min\limits_{1 \le \ell \le n_3} \|\hat{\tilde{A}}_{i:}^{(\ell)}\|_F^2}
\le
\frac{\gamma_i}{\|\tilde{\mathcal{A}}_{i::}\|_F^2}
\qquad \text{for all } k \in [n_3].
\]
Substituting this into the summation gives
\begin{align*}
\frac{1}{n_3} \sum_{k=1}^{n_3} \frac{\|\hat{N}_i^{(k)}\|_F^2}{\|\hat{\tilde{A}}_{i:}^{(k)}\|_F^2}
&\le
\frac{1}{n_3} \sum_{k=1}^{n_3}
\left( \frac{\gamma_i}{\|\tilde{\mathcal{A}}_{i::}\|_F^2} \right)
\|\hat{N}_i^{(k)}\|_F^2 \\
&=
\frac{\gamma_i}{\|\tilde{\mathcal{A}}_{i::}\|_F^2}
\left( \frac{1}{n_3} \sum_{k=1}^{n_3} \|\hat{N}_i^{(k)}\|_F^2 \right)
=
\gamma_i \frac{\|\mathcal{N}_i\|_F^2}{\|\tilde{\mathcal{A}}_{i::}\|_F^2},
\end{align*}
where we used Parseval's theorem,
\[
\|\mathcal N_i\|_F^2
=
\frac{1}{n_3}\sum_{k=1}^{n_3}\|\hat N_i^{(k)}\|_F^2.
\]

By the law of total expectation, we obtain the following. The summand below does not depend on \(\mathcal E^t\), so the outer expectation can be dropped after the first equality.
\begin{align*}
\mathbb{E}\left[
\left\|
\tilde{\mathcal{A}}_{i_t::}^*
(\tilde{\mathcal{A}}_{i_t::} \tilde{\mathcal{A}}_{i_t::}^*)^{-1}
\mathcal{N}_{i_t}
\right\|_F^2
\right]
&=
\sum_{i\in\mathcal I_+}
\frac{\|\tilde{\mathcal{A}}_{i::}\|_F^2}{\|\tilde{\mathcal{A}}\|_F^2}
\left\|
\tilde{\mathcal{A}}_{i::}^*
(\tilde{\mathcal{A}}_{i::}\tilde{\mathcal{A}}_{i::}^*)^{-1}
\mathcal{N}_i
\right\|_F^2 \\
&\le
\sum_{i\in\mathcal I_+}
\frac{\|\tilde{\mathcal{A}}_{i::}\|_F^2}{\|\tilde{\mathcal{A}}\|_F^2}
\gamma_i
\frac{\|\mathcal{N}_i\|_F^2}{\|\tilde{\mathcal{A}}_{i::}\|_F^2}
\\
&\le
\frac{\gamma}{\|\tilde{\mathcal{A}}\|_F^2}
\sum_{i\in\mathcal I_+}
\|\mathcal{N}_i\|_F^2 \\
&\le
\frac{\gamma}{\|\tilde{\mathcal{A}}\|_F^2}
\sum_{i=1}^{n_1}
\|\mathcal{N}_i\|_F^2 \\
&=
\frac{\gamma \|\mathcal E_A \mathcal{X}^\star - \mathcal E_B\|_F^2}
{\|\tilde{\mathcal{A}}\|_F^2}.
\end{align*}
Substituting the bounds for both terms back into \eqref{eq:expectation_decomposition} completes the proof.

\section{Proof of Theorem~\ref{thm:sdtrk_recursion}}\label{sec:sdtrk_recursion}

We first decompose the SD-TRK error update into a damped error-propagation term
and a noise-injection term. We then expand the squared Frobenius norm in the
Fourier domain, bound the cross term using Young's inequality, and finally average
over the randomized row selection.

By Lemma~\ref{lem:sdtrk_error_update}, the error tensor satisfies
\[
\mathcal{E}^{t+1}
=
\left(\mathcal{I}-\omega \tilde{\mathcal Q}_{i_t}^{\lambda}\right)\mathcal{E}^{t}
-
\omega
\tilde{\mathcal A}_{i_t::}^{*}
\left(
\tilde{\mathcal A}_{i_t::}
\tilde{\mathcal A}_{i_t::}^{*}
+\lambda \mathcal{I}
\right)^{-1}
\mathcal{N}_{i_t}.
\]
To estimate \(\|\mathcal{E}^{t+1}\|_F^2\), we pass to the Fourier domain, as in
the proof of Lemma~\ref{lem:orthogonality}. For any tensor \(\mathcal Y\), let
\(\hat{\mathcal Y}\) denote its discrete Fourier transform along the third mode,
and let \(\hat Y^{(k)}\) be its \(k\)-th frontal slice. By the properties of the
t-product in the Fourier domain (see Fact 1 in \cite{MaMolitor2022}), for
each \(k\in[n_3]\),
\[
\hat E^{t+1,(k)}
=
\left(
I-\omega \hat Q_{i_t}^{\lambda,(k)}
\right)
\hat E^{t,(k)}
-
\omega
\bigl(\hat{\tilde A}_{i_t:}^{(k)}\bigr)^*
\left(
\hat{\tilde A}_{i_t:}^{(k)}
\bigl(\hat{\tilde A}_{i_t:}^{(k)}\bigr)^*
+\lambda I
\right)^{-1}
\hat N_{i_t}^{(k)},
\]
where
\[
\hat Q_{i_t}^{\lambda,(k)}
:=
\bigl(\hat{\tilde A}_{i_t:}^{(k)}\bigr)^*
\left(
\hat{\tilde A}_{i_t:}^{(k)}
\bigl(\hat{\tilde A}_{i_t:}^{(k)}\bigr)^*
+\lambda I
\right)^{-1}
\hat{\tilde A}_{i_t:}^{(k)}.
\]
Here, \(\hat{\tilde A}_{i_t:}^{(k)}\) denotes the \(i_t\)-th row of the matrix \(\hat{\tilde A}^{(k)}\).

Now define
\[
U^{(k)}
:=
\left(
I-\omega \hat Q_{i_t}^{\lambda,(k)}
\right)
\hat E^{t,(k)},
\qquad
V^{(k)}
:=
-
\omega
\bigl(\hat{\tilde A}_{i_t:}^{(k)}\bigr)^*
\left(
\hat{\tilde A}_{i_t:}^{(k)}
\bigl(\hat{\tilde A}_{i_t:}^{(k)}\bigr)^*
+\lambda I
\right)^{-1}
\hat N_{i_t}^{(k)}.
\]
Then
\[
\hat E^{t+1,(k)}=U^{(k)}+V^{(k)}.
\]

We next expand the squared Frobenius norm of \(\hat E^{t+1,(k)}\) using the standard matrix Frobenius inner product
\[
\langle X,Y\rangle_F := \trace(Y^*X).
\]
Since \(\|X\|_F^2=\langle X,X\rangle_F\), we obtain
\[
\begin{aligned}
\bigl\|\hat E^{t+1,(k)}\bigr\|_F^2
&=
\|U^{(k)}+V^{(k)}\|_F^2 \\
&=
\langle U^{(k)}+V^{(k)},\,U^{(k)}+V^{(k)}\rangle_F \\
&=
\langle U^{(k)},U^{(k)}\rangle_F
+\langle U^{(k)},V^{(k)}\rangle_F
+\langle V^{(k)},U^{(k)}\rangle_F
+\langle V^{(k)},V^{(k)}\rangle_F \\
&=
\|U^{(k)}\|_F^2
+\langle U^{(k)},V^{(k)}\rangle_F
+\overline{\langle U^{(k)},V^{(k)}\rangle_F}
+\|V^{(k)}\|_F^2 \\
&=
\|U^{(k)}\|_F^2
+2\operatorname{Re}\langle U^{(k)},V^{(k)}\rangle_F
+\|V^{(k)}\|_F^2.
\end{aligned}
\]

Unlike the proof of Lemma~\ref{lem:orthogonality}, the cross term does not vanish here, because \(\hat Q_{i_t}^{\lambda,(k)}\) is not an orthogonal projection matrix in general when \(\lambda>0\). Therefore, we estimate the cross term by Young's inequality. For any \(\alpha>0\),
\[
2\operatorname{Re}\langle U^{(k)},V^{(k)}\rangle_F
\le
2\bigl|\langle U^{(k)},V^{(k)}\rangle_F\bigr|
\le
2\|U^{(k)}\|_F\|V^{(k)}\|_F
\le
\alpha \|U^{(k)}\|_F^2
+
\frac{1}{\alpha}\|V^{(k)}\|_F^2,
\]
where we used the Cauchy--Schwarz inequality in the second step and Young's inequality in the third step.

Substituting this estimate into the previous identity yields
\[
\begin{aligned}
\bigl\|\hat E^{t+1,(k)}\bigr\|_F^2
&\le
\|U^{(k)}\|_F^2
+
\alpha \|U^{(k)}\|_F^2
+
\|V^{(k)}\|_F^2
+
\frac{1}{\alpha}\|V^{(k)}\|_F^2 \\
&=
(1+\alpha)\|U^{(k)}\|_F^2
+
\left(1+\frac{1}{\alpha}\right)\|V^{(k)}\|_F^2.
\end{aligned}
\]
That is,
\[
\begin{aligned}
\bigl\|\hat E^{t+1,(k)}\bigr\|_F^2
&\le
(1+\alpha)
\left\|
\left(
I-\omega \hat Q_{i_t}^{\lambda,(k)}
\right)
\hat E^{t,(k)}
\right\|_F^2 \\
&\quad
+
\omega^2
\left(1+\frac{1}{\alpha}\right)
\left\|
\bigl(\hat{\tilde A}_{i_t:}^{(k)}\bigr)^*
\left(
\hat{\tilde A}_{i_t:}^{(k)}
\bigl(\hat{\tilde A}_{i_t:}^{(k)}\bigr)^*
+\lambda I
\right)^{-1}
\hat N_{i_t}^{(k)}
\right\|_F^2 .
\end{aligned}
\]

We now sum this inequality over all \(k\in[n_3]\) and divide by \(n_3\). By Parseval's identity,
\[
\|\mathcal Y\|_F^2
=
\frac{1}{n_3}\sum_{k=1}^{n_3}\|\hat Y^{(k)}\|_F^2
\]
for every tensor \(\mathcal Y\). Applying this identity to each term above gives
\[
\begin{aligned}
\|\mathcal E^{t+1}\|_F^2
&=
\frac{1}{n_3}\sum_{k=1}^{n_3}
\bigl\|\hat E^{t+1,(k)}\bigr\|_F^2 \\
&\le
(1+\alpha)
\frac{1}{n_3}\sum_{k=1}^{n_3}
\left\|
\left(
I-\omega \hat Q_{i_t}^{\lambda,(k)}
\right)
\hat E^{t,(k)}
\right\|_F^2 \\
&\quad
+
\omega^2
\left(1+\frac{1}{\alpha}\right)
\frac{1}{n_3}\sum_{k=1}^{n_3}
\left\|
\bigl(\hat{\tilde A}_{i_t:}^{(k)}\bigr)^*
\left(
\hat{\tilde A}_{i_t:}^{(k)}
\bigl(\hat{\tilde A}_{i_t:}^{(k)}\bigr)^*
+\lambda I
\right)^{-1}
\hat N_{i_t}^{(k)}
\right\|_F^2 \\
&=
(1+\alpha)
\left\|
\left(\mathcal I-\omega \tilde{\mathcal Q}_{i_t}^{\lambda}\right)\mathcal E^t
\right\|_F^2 \\
&\quad
+
\omega^2
\left(1+\frac{1}{\alpha}\right)
\left\|
\tilde{\mathcal A}_{i_t::}^{*}
\left(
\tilde{\mathcal A}_{i_t::}
\tilde{\mathcal A}_{i_t::}^{*}
+\lambda \mathcal I
\right)^{-1}
\mathcal N_{i_t}
\right\|_F^2.
\end{aligned}
\]
Hence,
\[
\|\mathcal E^{t+1}\|_F^2
\le
(1+\alpha)
\left\|
\left(\mathcal I-\omega \tilde{\mathcal Q}_{i_t}^{\lambda}\right)\mathcal E^t
\right\|_F^2
+
\omega^2
\left(1+\frac{1}{\alpha}\right)
\left\|
\tilde{\mathcal A}_{i_t::}^{*}
\left(
\tilde{\mathcal A}_{i_t::}
\tilde{\mathcal A}_{i_t::}^{*}
+\lambda \mathcal I
\right)^{-1}
\mathcal N_{i_t}
\right\|_F^2.
\]

Taking the conditional expectation with respect to the sampled index \(i_t\), conditioned on \(\mathcal E^t\), we obtain
\[
\begin{aligned}
\mathbb{E}\!\left[
\|\mathcal{E}^{t+1}\|_F^2
\mid
\mathcal{E}^{t}
\right]
&\le
(1+\alpha)
\mathbb{E}\!\left[
\left\|
\left(\mathcal{I}-\omega \tilde{\mathcal Q}_{i_t}^{\lambda}\right)\mathcal{E}^{t}
\right\|_F^2
\mid
\mathcal{E}^{t}
\right]  \\
&\quad
+
\omega^2\left(1+\frac{1}{\alpha}\right)
\mathbb{E}\!\left[
\left\|
\tilde{\mathcal A}_{i_t::}^{*}
\left(
\tilde{\mathcal A}_{i_t::}
\tilde{\mathcal A}_{i_t::}^{*}
+\lambda \mathcal{I}
\right)^{-1}
\mathcal{N}_{i_t}
\right\|_F^2
\mid
\mathcal{E}^{t}
\right].
\end{aligned}
\]
It remains to estimate the two conditional expectation terms on the right-hand side separately.

By the sampling rule
\[
p_i=\frac{\|\tilde{\mathcal A}_{i::}\|_F^2}{\|\tilde{\mathcal A}\|_F^2},
\qquad i\in[n_1],
\]
the conditional expectation with respect to the sampled index \(i_t\), conditioned on \(\mathcal E^t\), can be written explicitly as follows. Terms with \(p_i=0\) do not contribute to the expectation, so it is enough to sum over \(i\in\mathcal I_+\), where \(\mathcal I_+\) is defined in Theorem~\ref{thm:upper_recursion}. For \(i\in\mathcal I_+\), the operator \(\tilde{\mathcal Q}_i^\lambda\) is well-defined under Assumption~\ref{assump:invertibility}; for \(\lambda>0\), the regularization term also prevents singular normalization factors. We obtain
\[
\begin{aligned}
\mathbb{E}\!\left[
\|\mathcal{E}^{t+1}\|_F^2
\mid
\mathcal{E}^{t}
\right]
&\le
(1+\alpha)
\sum_{i\in\mathcal I_+}
\frac{\|\tilde{\mathcal A}_{i::}\|_F^2}
{\|\tilde{\mathcal A}\|_F^2}
\left\|
\left(\mathcal{I}-\omega \tilde{\mathcal Q}_{i}^\lambda\right)\mathcal{E}^{t}
\right\|_F^2  \\
&\quad
+
\omega^2\left(1+\frac{1}{\alpha}\right)
\sum_{i\in\mathcal I_+}
\frac{\|\tilde{\mathcal A}_{i::}\|_F^2}
{\|\tilde{\mathcal A}\|_F^2}
\left\|
\tilde{\mathcal A}_{i::}^{*}
\left(
\tilde{\mathcal A}_{i::}
\tilde{\mathcal A}_{i::}^{*}
+\lambda \mathcal{I}
\right)^{-1}
\mathcal N_i
\right\|_F^2 .
\end{aligned}
\]

We first estimate the contraction term
\[
\sum_{i\in\mathcal I_+}
\frac{\|\tilde{\mathcal A}_{i::}\|_F^2}
{\|\tilde{\mathcal A}\|_F^2}
\left\|
\left(\mathcal I-\omega \tilde{\mathcal Q}_{i}^\lambda\right)\mathcal E^t
\right\|_F^2.
\]
As in the proof of Lemma~\ref{lem:orthogonality}, we pass to the Fourier domain.
For each \(k\in[n_3]\), the tensor operator \(\tilde{\mathcal Q}_{i}^{\lambda}\)
corresponds to the matrix
\[
\hat Q_i^{\lambda,(k)}
=
\bigl(\hat{\tilde A}_{i:}^{(k)}\bigr)^*
\left(
\hat{\tilde A}_{i:}^{(k)}
\bigl(\hat{\tilde A}_{i:}^{(k)}\bigr)^*
+\lambda I
\right)^{-1}
\hat{\tilde A}_{i:}^{(k)} .
\]
Here, \(\hat{\tilde A}_{i:}^{(k)}\) is the \(i\)-th row of the matrix
\(\hat{\tilde A}^{(k)}\). For notational simplicity, fix \(i\) and \(k\), and write
\[
a:=\hat{\tilde A}_{i:}^{(k)}\in\mathbb C^{1\times n_2},
\qquad
x:=\hat E^{t,(k)}\in\mathbb C^{n_2\times n_4},
\qquad
Q:=\frac{a^*a}{\|a\|_2^2+\lambda}.
\]
Then
\[
\left\|
\left(I-\omega Q\right)x
\right\|_F^2
=
\trace
\bigl(
x^*(I-\omega Q)^*(I-\omega Q)x
\bigr).
\]
Since \(Q^*=Q\), we obtain
\[
\begin{aligned}
\left\|
\left(I-\omega Q\right)x
\right\|_F^2
&=
\trace
\bigl(
x^*(I-\omega Q^*-\omega Q+\omega^2Q^*Q)x
\bigr) \\
&=
\|x\|_F^2
-2\omega\,\trace(x^*Qx)
+\omega^2\trace(x^*Q^2x).
\end{aligned}
\]
Moreover,
\[
\trace(x^*Qx)
=
\frac{\|ax\|_F^2}{\|a\|_2^2+\lambda},
\]
and, since
\[
Q^2
=
\frac{\|a\|_2^2}{\|a\|_2^2+\lambda}\,Q,
\]
we also have
\[
\trace(x^*Q^2x)
=
\frac{\|a\|_2^2}{(\|a\|_2^2+\lambda)^2}\|ax\|_F^2.
\]
Substituting these identities gives
\[
\begin{aligned}
\left\|
\left(I-\omega Q\right)x
\right\|_F^2
&=
\|x\|_F^2
-
\frac{2\omega}{\|a\|_2^2+\lambda}\|ax\|_F^2
+
\frac{\omega^2\|a\|_2^2}{(\|a\|_2^2+\lambda)^2}\|ax\|_F^2 \\
&=
\|x\|_F^2
-
\frac{(2\omega-\omega^2)\|a\|_2^2+2\omega\lambda}
{(\|a\|_2^2+\lambda)^2}
\|ax\|_F^2.
\end{aligned}
\]
Let
\[
S:=\sigma_{\max}^2(\bcirc(\tilde{\mathcal A})),
\qquad
s:=\|a\|_2^2 .
\]
Then \(0\le s\le S\), since
\[
\|a\|_2
\le
\sigma_{\max}\bigl(\hat{\tilde A}^{(k)}\bigr)
\le
\sigma_{\max}(\bcirc(\tilde{\mathcal A})),
\]
where the second inequality follows from the block-diagonalization of \(\bcirc(\tilde{\mathcal A})\) in Appendix~\ref{sec:appendix_fft}. Define
\[
c(s):=
\frac{(2\omega-\omega^2)s+2\omega\lambda}{(s+\lambda)^2}.
\]
Since
\[
c'(s)
=
-\frac{(2\omega-\omega^2)s+(2\omega+\omega^2)\lambda}{(s+\lambda)^3}
\le 0,
\]
the function \(c(s)\) is nonincreasing on its domain \(s+\lambda>0\); in particular, for the admissible Fourier row norms considered here, \(c(s)\ge c(S)\). 
Hence
\[
c(s)\ge c(S)
=
\frac{(2\omega-\omega^2)S+2\omega\lambda}{(S+\lambda)^2}
\ge
\frac{2\omega-\omega^2}{S+\lambda}.
\]
Therefore,
\[
\left\|
\left(I-\omega Q\right)x
\right\|_F^2
\le
\|x\|_F^2
-
\frac{2\omega-\omega^2}
{\sigma_{\max}^2(\bcirc(\tilde{\mathcal A}))+\lambda}
\|ax\|_F^2.
\]
Therefore, for each frequency slice \(k\),
\[
\left\|
\left(I-\omega \hat Q_i^{\lambda,(k)}\right)\hat E^{t,(k)}
\right\|_F^2
\le
\left\|
\hat E^{t,(k)}
\right\|_F^2
-
\frac{2\omega-\omega^2}
{\sigma_{\max}^2(\bcirc(\tilde{\mathcal A}))+\lambda}
\left\|
\hat{\tilde A}_{i:}^{(k)}
\hat E^{t,(k)}
\right\|_F^2.
\]

Applying Parseval's identity to the slice-wise estimate obtained above, we obtain, for each \(i\in\mathcal I_+\),
\[
\left\|
\left(\mathcal I-\omega \tilde{\mathcal Q}_{i}^\lambda\right)\mathcal E^t
\right\|_F^2
\le
\|\mathcal E^t\|_F^2
-
\frac{2\omega-\omega^2}
{\sigma_{\max}^2(\bcirc(\tilde{\mathcal A}))+\lambda}
\left\|
\tilde{\mathcal A}_{i::}\mathcal E^t
\right\|_F^2 .
\]
Multiplying by \(p_i=\|\tilde{\mathcal A}_{i::}\|_F^2/\|\tilde{\mathcal A}\|_F^2\) and summing over \(i\in\mathcal I_+\), we get the bound below. Note that rows with \(p_i=0\) satisfy \(\tilde{\mathcal A}_{i::}\mathcal E^t=0\), so extending the sum on the right-hand side to all \(i\in[n_1]\) does not change its value.
\[
\begin{aligned}
\sum_{i\in\mathcal I_+}
\frac{\|\tilde{\mathcal A}_{i::}\|_F^2}
{\|\tilde{\mathcal A}\|_F^2}
\left\|
\left(\mathcal I-\omega \tilde{\mathcal Q}_{i}^\lambda\right)\mathcal E^t
\right\|_F^2
&\le
\|\mathcal E^t\|_F^2
-
\frac{2\omega-\omega^2}
{\sigma_{\max}^2(\bcirc(\tilde{\mathcal A}))+\lambda}
\sum_{i=1}^{n_1}
\frac{\|\tilde{\mathcal A}_{i::}\|_F^2}
{\|\tilde{\mathcal A}\|_F^2}
\left\|
\tilde{\mathcal A}_{i::}\mathcal E^t
\right\|_F^2 .
\end{aligned}
\]
Therefore,
\[
\mathbb E\!\left[
\left\|
\left(\mathcal I-\omega \tilde{\mathcal Q}_{i_t}^{\lambda}\right)\mathcal E^t
\right\|_F^2
\,\middle|\,\mathcal E^t
\right]
\le
\|\mathcal E^t\|_F^2
-
\frac{2\omega-\omega^2}
{\sigma_{\max}^2(\bcirc(\tilde{\mathcal A}))+\lambda}
\sum_{i=1}^{n_1}
\frac{\|\tilde{\mathcal A}_{i::}\|_F^2}
{\|\tilde{\mathcal A}\|_F^2}
\left\|
\tilde{\mathcal A}_{i::}\mathcal E^t
\right\|_F^2 .
\]
To close this estimate, it remains to bound the weighted term on the right-hand side from below.

We next estimate the perturbation term
\[
\sum_{i\in\mathcal I_+}
\frac{\|\tilde{\mathcal A}_{i::}\|_F^2}
{\|\tilde{\mathcal A}\|_F^2}
\left\|
\tilde{\mathcal A}_{i::}^{*}
\left(
\tilde{\mathcal A}_{i::}
\tilde{\mathcal A}_{i::}^{*}
+\lambda \mathcal I
\right)^{-1}
\mathcal N_i
\right\|_F^2 .
\]
For each \(i\in\mathcal I_+\), define
\[
\mathcal M_i^\lambda
:=
\tilde{\mathcal A}_{i::}^{*}
\left(
\tilde{\mathcal A}_{i::}
\tilde{\mathcal A}_{i::}^{*}
+\lambda \mathcal I
\right)^{-1}.
\]
In the Fourier domain, the \(k\)-th frontal slice corresponding to
\(\mathcal M_i^\lambda\) is denoted by
\[
\hat M_i^{\lambda,(k)}
:=
\hat{\tilde A}_{i:}^{(k)*}
\left(
\hat{\tilde A}_{i:}^{(k)}
\hat{\tilde A}_{i:}^{(k)*}
+\lambda I
\right)^{-1}.
\]
Then, by Parseval's identity and the matrix Frobenius submultiplicative property in each Fourier slice,
\[
\begin{aligned}
\left\|\mathcal M_i^\lambda \mathcal N_i\right\|_F^2
&=
\frac{1}{n_3}
\sum_{k=1}^{n_3}
\left\|
\hat M_i^{\lambda,(k)}\hat N_i^{(k)}
\right\|_F^2  \\
&\le
\frac{1}{n_3}
\sum_{k=1}^{n_3}
\left\|
\hat M_i^{\lambda,(k)}
\right\|_F^2
\left\|
\hat N_i^{(k)}
\right\|_F^2  \\
&\le
\frac{1}{n_3}
\left(
\sum_{k=1}^{n_3}
\left\|
\hat M_i^{\lambda,(k)}
\right\|_F^2
\right)
\left(
\sum_{k=1}^{n_3}
\left\|
\hat N_i^{(k)}
\right\|_F^2
\right)  \\
&=
n_3
\left\|\mathcal M_i^\lambda\right\|_F^2
\left\|\mathcal N_i\right\|_F^2 .
\end{aligned}
\]
Therefore,
\[
\begin{aligned}
&\sum_{i\in\mathcal I_+}
\frac{\|\tilde{\mathcal A}_{i::}\|_F^2}
{\|\tilde{\mathcal A}\|_F^2}
\left\|
\tilde{\mathcal A}_{i::}^{*}
\left(
\tilde{\mathcal A}_{i::}
\tilde{\mathcal A}_{i::}^{*}
+\lambda \mathcal I
\right)^{-1}
\mathcal N_i
\right\|_F^2 \\
&\le
n_3
\sum_{i\in\mathcal I_+}
\frac{\|\tilde{\mathcal A}_{i::}\|_F^2}
{\|\tilde{\mathcal A}\|_F^2}
\left\|
\tilde{\mathcal A}_{i::}^{*}
\left(
\tilde{\mathcal A}_{i::}
\tilde{\mathcal A}_{i::}^{*}
+\lambda \mathcal I
\right)^{-1}
\right\|_F^2
\|\mathcal N_i\|_F^2 \\
&\le
n_3
\left(
\max_{i\in\mathcal I_+}
\left\|
\tilde{\mathcal A}_{i::}^{*}
\left(
\tilde{\mathcal A}_{i::}
\tilde{\mathcal A}_{i::}^{*}
+\lambda \mathcal I
\right)^{-1}
\right\|_F^2
\right)
\sum_{i=1}^{n_1}\|\mathcal N_i\|_F^2 .
\end{aligned}
\]
Recalling the definition
\[
\eta(\omega,\lambda)
:=
n_3\,
\omega^2
\max_{i\in\mathcal I_+}
\left\|
\tilde{\mathcal A}_{i::}^{*}
\left(
\tilde{\mathcal A}_{i::}
\tilde{\mathcal A}_{i::}^{*}
+\lambda \mathcal I
\right)^{-1}
\right\|_F^2,
\]
we obtain
\[
\begin{aligned}
&\omega^2\left(1+\frac{1}{\alpha}\right)
\sum_{i\in\mathcal I_+}
\frac{\|\tilde{\mathcal A}_{i::}\|_F^2}
{\|\tilde{\mathcal A}\|_F^2}
\left\|
\tilde{\mathcal A}_{i::}^{*}
\left(
\tilde{\mathcal A}_{i::}
\tilde{\mathcal A}_{i::}^{*}
+\lambda \mathcal I
\right)^{-1}
\mathcal N_i
\right\|_F^2 \\
&\le
\left(1+\frac{1}{\alpha}\right)
\eta(\omega,\lambda)
\sum_{i=1}^{n_1}\|\mathcal N_i\|_F^2.
\end{aligned}
\]
Finally, if we define the tensor
\[
\mathcal N:=\mathcal E_A\mathcal X^\star-\mathcal E_B,
\]
then its \(i\)-th horizontal slice is \(\mathcal N_i\), and hence
\[
\sum_{i=1}^{n_1}\|\mathcal N_i\|_F^2
=
\|\mathcal N\|_F^2
=
\|\mathcal E_A\mathcal X^\star-\mathcal E_B\|_F^2.
\]
Therefore,
\[
\omega^2\left(1+\frac{1}{\alpha}\right)
\sum_{i\in\mathcal I_+}
\frac{\|\tilde{\mathcal A}_{i::}\|_F^2}
{\|\tilde{\mathcal A}\|_F^2}
\left\|
\tilde{\mathcal A}_{i::}^{*}
\left(
\tilde{\mathcal A}_{i::}
\tilde{\mathcal A}_{i::}^{*}
+\lambda \mathcal I
\right)^{-1}
\mathcal N_i
\right\|_F^2
\le
\left(1+\frac{1}{\alpha}\right)
\eta(\omega,\lambda)
\|\mathcal E_A \mathcal{X}^\star-\mathcal E_B\|_F^2.
\]

Thus, combining the above estimates and applying the tower property of conditional expectation, we obtain
\[
\begin{aligned}
\mathbb E\!\left[\|\mathcal E^{t+1}\|_F^2\right]
&\le
(1+\alpha)\mathbb E\!\left[\|\mathcal E^{t}\|_F^2\right] \\
&\quad
-
(1+\alpha)
\frac{2\omega-\omega^2}
{\sigma_{\max}^2(\bcirc(\tilde{\mathcal A}))+\lambda}
\,
\mathbb E\!\left[
\sum_{i=1}^{n_1}
\frac{\|\tilde{\mathcal A}_{i::}\|_F^2}
{\|\tilde{\mathcal A}\|_F^2}
\left\|
\tilde{\mathcal A}_{i::}\mathcal E^t
\right\|_F^2
\right] \\
&\quad
+
\left(1+\frac{1}{\alpha}\right)
\eta(\omega,\lambda)
\|\mathcal E_A \mathcal{X}^\star-\mathcal E_B\|_F^2.
\end{aligned}
\]
To complete the proof, it remains to bound the weighted expectation term in the second line from below.

To bound the weighted term from below, define
\[
\underline p
:=
\min_{i\in[n_1]:\,p_i>0}
\frac{\|\tilde{\mathcal A}_{i::}\|_F^2}{\|\tilde{\mathcal A}\|_F^2}.
\]
Then
\[
\sum_{i=1}^{n_1}
\frac{\|\tilde{\mathcal A}_{i::}\|_F^2}{\|\tilde{\mathcal A}\|_F^2}
\left\|
\tilde{\mathcal A}_{i::}\mathcal E^t
\right\|_F^2
\ge
\underline p
\sum_{i=1}^{n_1}
\left\|
\tilde{\mathcal A}_{i::}\mathcal E^t
\right\|_F^2.
\]
Moreover,
\[
\sum_{i=1}^{n_1}
\left\|
\tilde{\mathcal A}_{i::}\mathcal E^t
\right\|_F^2
=
\|\tilde{\mathcal A}\mathcal E^t\|_F^2
\ge
\sigma_{\min}^2(\bcirc(\tilde{\mathcal A}))
\|\mathcal E^t\|_F^2.
\]
Hence
\[
\sum_{i=1}^{n_1}
\frac{\|\tilde{\mathcal A}_{i::}\|_F^2}{\|\tilde{\mathcal A}\|_F^2}
\left\|
\tilde{\mathcal A}_{i::}\mathcal E^t
\right\|_F^2
\ge
\underline p\,
\sigma_{\min}^2(\bcirc(\tilde{\mathcal A}))
\|\mathcal E^t\|_F^2.
\]
Substituting this into the previous estimate yields
\[
\mathbb E\!\left[\|\mathcal E^{t+1}\|_F^2\right]
\le
\rho_\alpha(\omega,\lambda)\,
\mathbb E\!\left[\|\mathcal E^{t}\|_F^2\right]
+
\eta_\alpha(\omega,\lambda)
\|\mathcal E_A \mathcal{X}^\star-\mathcal E_B\|_F^2,
\]
where
\[
\rho_\alpha(\omega,\lambda)
:=
(1+\alpha)
\left(
1-
\underline p\,
\frac{(2\omega-\omega^2)\sigma_{\min}^2(\bcirc(\tilde{\mathcal A}))}
{\sigma_{\max}^2(\bcirc(\tilde{\mathcal A}))+\lambda}
\right),
\]
and
\[
\eta_\alpha(\omega,\lambda)
:=
\left(1+\frac{1}{\alpha}\right)\eta(\omega,\lambda).
\]
This completes the proof.

\subsection{Remark}

Two points clarify the proof strategy above. First, unlike the standard TRK case in Lemma~\ref{lem:orthogonality},
the regularized tensor operator \(\tilde{\mathcal Q}_i^\lambda\), equivalently
its Fourier-slice matrix representation \(\hat Q_i^{\lambda,(k)}\), is not an
orthogonal projection when \(\lambda>0\).
Consequently, the cross term in the expansion of \(\|\mathcal E^{t+1}\|_F^2\) cannot be eliminated by an orthogonality argument and is instead controlled by Young's inequality.

Second, although one might try to remove the cross term by taking expectation and using the noise structure, this is not immediate under the sampling rule
\[
p_i=\frac{\|\tilde{\mathcal A}_{i::}\|_F^2}{\|\tilde{\mathcal A}\|_F^2},
\]
since the sampling probabilities themselves depend on the observed noisy tensor \(\tilde{\mathcal A}\). In addition, the bound involving \(\sigma_{\max}^2(\bcirc(\tilde{\mathcal A}))\) follows from the fact that each Fourier-slice row norm is bounded by the corresponding global spectral quantity of \(\bcirc(\tilde{\mathcal A})\).

\end{document}